\documentclass[a4paper,10pt]{article}

\usepackage[pdftex,a4paper,body={15cm,24cm}]{geometry}
\usepackage[T1]{fontenc}
\usepackage{pgfplots}
\usepackage[latin1,utf8x]{inputenc}
\usepackage[cm]{aeguill}
\usepackage[english]{babel}
\usepackage{amsfonts,amssymb,amsmath,amsthm}
\usepackage{latexsym}
\usepackage{eurosym}
\usepackage{subfigure}
\usepackage{graphicx}
\usepackage{bbm,dsfont}
\usepackage{cases}
\usepackage{chngcntr}
\usepackage{mathrsfs}
\usepackage{xcolor}

\pgfplotsset{compat=newest, ticks=none}
\usepgfplotslibrary{fillbetween}

\DeclareGraphicsExtensions{.png,.jpg,.bmp,.eps,.pdf}

\numberwithin{equation}{section}
\numberwithin{figure}{section}

\newtheorem{thm}{Theorem}[section]
\newtheorem{prop}[thm]{Proposition}
\newtheorem{defn}[thm]{Definition}
\newtheorem{lemma}[thm]{Lemma}

\newtheorem{remark}[thm]{Remark}

\DeclareMathOperator*{\argmax}{arg\,max}

\newcommand{\nn}{\mathbb{N}}
\newcommand{\rr}{\mathbb{R}}

\newcommand{\eee}{\mathbb{E}_{x}}
\newcommand{\aaa}{\mathcal{A}}
\newcommand{\mm}{\mathcal{M}}
\newcommand{\hh}{\mathcal{H}}

\newcommand{\vphi}{\varphi}

\newcommand{\xx}{\bar x}

\newcommand{\q}{\theta}

\newcommand{\1}{\bar x_1 }
\newcommand{\2}{\bar x_2 }
\newcommand{\3}{x^*_1 }
\newcommand{\4}{x^*_2 }

\title{Nonzero-sum stochastic differential games with impulse controls: a verification theorem with applications}

\author{
	René Aïd\thanks{Economics Department (LEDa), University Paris Dauphine, and Finance For Energy Market Research Centre (FiME).}
	\qquad Matteo Basei\footnote{Department of Industrial Engineering and Operations Research (IEOR), University of California, Berkeley. E-mail: basei@berkeley.edu} 
	\qquad Giorgia Callegaro\footnote{Department of Mathematics, University of Padova.}\\ 
	Luciano Campi\thanks{Department of Statistics, London School of Economics and Political Science.}
	\qquad Tiziano Vargiolu\footnotemark[3]}

\date{\today}

\begin{document}

\maketitle

\vspace{1cm}

\begin{abstract}
	We consider a general nonzero-sum impulse game with two players. The main mathematical contribution of the paper is a verification theorem which provides, under some regularity conditions, a suitable system of quasi-variational inequalities for the payoffs and the strategies of the two players at some Nash equilibrium. As an application, we study an impulse game with a one-dimensional state variable, following a real-valued scaled Brownian motion, and two players with linear and symmetric running payoffs. We fully characterize a family of Nash equilibria and provide explicit expressions for the corresponding equilibrium strategies and payoffs. We also prove some asymptotic results with respect to the intervention costs. Finally, we consider two further non-symmetric examples where a Nash equilibrium is found numerically.
	\medskip \\
	{\color{black}\textbf{AMS classification:} 91A15, 91B70, 93E20.}
\end{abstract}

\vspace{1cm}

\section{Introduction}
In this article, we study a general two-player nonzero-sum stochastic differential game with impulse controls. In few words, after setting the general framework, we focus on the notion of Nash equilibrium and identify the corresponding system of quasi-variational inequalities (QVIs). As an application, we consider an impulse game with a one-dimensional state variable and fully solve the system of QVIs, obtaining explicit expressions for the equilibrium payoffs and the corresponding strategies. This paper represents an extension of the results in the Ph.D. thesis \cite{Basei}. 

\vspace{0.4cm}

More specifically, we consider a game where two players can affect a continuous-time stochastic process $X$ by discrete-time interventions which consist in shifting $X$ to a new state. When none of the players intervenes, we assume $X$ to diffuse according to a standard stochastic differential equation. Each intervention corresponds to a cost for the intervening player and a gain for the opponent. The strategy of player $i \in \{1,2\}$ is determined by a couple $\vphi_i =(\mathcal{C}_i,\xi_i)$, where $\mathcal{C}_i$ is a fixed open subset and $\xi_i$ is a continuous function: namely, player $i$ intervenes if and only if the process $X$ exits from $\mathcal{C}_i$ and, when this happens, she shifts the process from state $x$ to state $\xi_i(x)$. {\color{black} If both the players want to intervene, we assume that player 1 has the priority over player 2.} Once the strategies $\vphi_i=(\mathcal{C}_i,\xi_i)$, $i \in \{1,2\}$, and a starting point $x$ have been chosen, a couple of impulse controls $\{ (\tau_{i,k},\delta_{i,k})\}_{k \geq 1}$ is uniquely defined: $\tau_{i,k} $ is the $k$-th intervention time of player $i$ and $\delta_{i,k}$ is the corresponding impulse. Each player aims at maximizing her payoff, defined as follows: for every $x$ belonging to some fixed subset $S \subseteq \rr^d$ and every couple of strategies $(\vphi_1,\vphi_2)$, we set 
\begin{multline}
	\label{zIntro3}
	J^i(x;\vphi_1,\vphi_2) := 
	\eee \bigg[ \int_0^{\tau_S} e^{-\rho_i s} f_i(X_s) ds 
	+ \sum_{ k \geq 1 \, : \, \tau_{i,k} < \tau_S } e^{-\rho_i \tau_{i,k}} \phi_{i} \Big( X_{(\tau_{i,k})^-}, \delta_{i,k} \Big)
	\\
	+ \sum_{ k \geq 1 \, : \, \tau_{j,k} < \tau_S } e^{-\rho_i \tau_{j,k}} \psi_{i} \Big( X_{(\tau_{j,k})^-}, \delta_{j,k} \Big)
	+ e^{-\rho_i \tau_S} h_i (X_{\tau_S})\mathbbm{1}_{ \{ \tau_S < +\infty \} } \bigg],
\end{multline}
where $i,j \in \{1,2\}$, $i \neq j$ and $\tau_S$ is the exit time of $X$ from $S$. The couple $(\vphi_1^*,\vphi_2^*)$ is a Nash equilibrium if
\begin{equation}
J^1(x;\vphi_1^*,\vphi_2^*) \geq J^1(x;\vphi_1,\vphi_2^*), \qquad\text{and}\qquad
J^2(x;\vphi_1^*,\vphi_2^*) \geq J^2(x;\vphi_1^*,\vphi_2),
\end{equation}
for every couple of strategies $\vphi_1,\vphi_2$.

\vspace{0.4cm} 

The first contribution of our paper is the Verification Theorem \ref{thm:verification}, {\color{black} which connects  the game in \eqref{zIntro3} to a suitable system of QVIs. To the best of our knowledge, nonzero-sum games with impulse controls have never been considered from a QVI perspective before.}

Namely, in Theorem \ref{thm:verification} we consider the following system of QVIs:
\begin{equation}
\label{zIntro4}
\begin{aligned}
& V_i = h_i,  && \text{in} \,\,\, \partial S,  \\
& \mm_jV_j - V_j \leq 0, && \text{in} \,\,\, S,   \\
& \hh_iV_i-V_i=0, && \text{in} \,\,\, \{\mm_jV_j - V_j = 0\},  \\
& \max\big\{\aaa V_i -\rho_i V_i + f_i, \mm_iV_i-V_i \}=0, && \text{in} \,\,\, \{\mm_jV_j - V_j < 0\},
\end{aligned}
\end{equation} 
where $i,j \in \{1,2\}$, $i \neq j$, \textcolor{black}{$\aaa$ is the infinitesimal generator of the uncontrolled state process} and $\mm_i, \hh_i$ are suitable intervention operators defined in Section \ref{ssec:QVI}. If two functions $V_i$, with $i \in \{1,2\}$, are a solution to \eqref{zIntro4}, have polynomial growth and satisfy the regularity condition 
\begin{equation}
	\label{zIntro5}
	V_i \in C^2(\mathcal D_j \setminus \partial \mathcal D_i) \cap C^1(\mathcal D_j) \cap C(\overline S),
\end{equation}
where $j \in \{1,2\}$ with $j \neq i$ and $\mathcal D_j=\{\mm_jV_j - V_j < 0\}$, then they coincide with some Nash equilibrium payoffs and a characterization of the corresponding equilibrium strategy is possible. 

\vspace{0.4cm}  

Our second contribution to this stream of research consists in providing examples of solvable impulse games. Using the Verification Theorem \ref{thm:verification} described above and solving the system of QVIs (\ref{zIntro4}), we are able to characterize the Nash equilibria. The main example is described in Sections \ref{sec:thepb}-\ref{ssec:asymp}, where (\ref{zIntro4}) is analytically solved and explicit formulas are provided. In Section \ref{ssec:moreexamples} we consider further families of problems, where (\ref{zIntro4}) is solved numerically. To our knowledge, these are the first examples of solvable nonzero-sum impulse games.

In Sections \ref{sec:thepb} we consider a two-player impulse game with a one-dimensional state variable $X$, modelled by a real-valued (scaled) Brownian motion. The two players have symmetric linear running payoffs and they can intervene on $X$ by shifting it from its current state, say $x$, to some other state $x+\delta$, with $\delta \in \rr$. When a player intervenes, she faces a penalty while her opponent faces a gain, both consisting in a fixed and in a variable part, which is assumed proportional to the size of the impulse. Hence, the players objective functions are
\begin{gather*}
J^1(x;\vphi_1,\vphi_2) \! := \! 
\eee \bigg[ \! \int_0^{\infty} \!\! e^{-\rho s} (X_s \!-\! s_1) ds  
\!-\! \sum_{k \geq 1} e^{-\rho \tau_{1,k}} (c \!+\! \lambda |\delta_{1,k}|)
\!+\! \sum_{k \geq 1} e^{-\rho \tau_{2,k}} (\tilde c \!+\! \tilde \lambda |\delta_{2,k}|)\bigg]\!,
\\
J^2(x;\vphi_1,\vphi_2) \!:=\! 
\eee \bigg[ \! \int_0^{\infty} \!\! e^{-\rho s} (s_2 \!-\! X_s) ds 
\!-\! \sum_{k \geq 1} e^{-\rho \tau_{2,k}} (c \!+\! \lambda |\delta_{2,k}|)
\!+\! \sum_{k \geq 1} e^{-\rho \tau_{1,k}} (\tilde c \!+\! \tilde \lambda |\delta_{1,k}|)\bigg]\!,
\end{gather*}
where $\{(\tau_{i,k},\delta_{i,k})\}_{k \geq 1}$ denotes the impulse control of player $i$ associated to the strategies $\vphi_1,\vphi_2$. Some preliminary heuristics on the QVIs in \eqref{zIntro4} leads us to consider a pair of candidates for the functions $V_i$. Then, a careful application of the verification theorem shows that such candidates actually coincide with the payoff functions of some Nash equilibrium. In particular, a practical characterization of the associated Nash equilibria is possible: player 1 (resp.~player 2) intervenes when the state $X$ is smaller than $\bar x_1$ (resp.~greater than $\bar x_2$) and moves the process to $x^*_1$ (resp.~$x^*_2$), for suitable $\bar x_i, x^*_i$. We provide explicit expressions for the payoff functions and for the parameters $\bar x_i, x^*_i$. Finally, we study the behaviour of the intervention region in some limit cases. In particular, we remark that in the case where $c = \tilde c$ and $\lambda = \tilde \lambda$, the game does not have an admissible Nash equilibrium. 

Finally, in Section \ref{ssec:moreexamples} we consider two further families of examples, with cubic payoffs and with linear and cubic payoffs. We adapt the technique described above and characterize a Nash equilibrium by solving the system of QVIs numerically. 

\vspace{0.4cm}

{\color{black} In applications where controllers act on the underlying process by discrete-time interventions, impulse controls offer a more realistic model with respect to standard continuous-time controls, even though at the cost of a more challenging technical framework. Quite surprisingly, however, the case of nonzero-sum games with impulse controls did not deserve enough attention so far. 

For a comprehensive introduction to impulse controls and the corresponding single-player control problems, we refer the reader to \cite{OksendalSulem07}. Among recent works on single-player impulse control problems, we cite \cite{BelakChristSeif17} for a general verification result,  \cite{AzimBayLab17,AzimFor16} for computational schemes for Hamilton-Jacobi-Bellman QVIs, \cite{AltarReppenSoner17} for the Merton problem with fixed and proportional transaction costs. As for stopping games treated with QVI techniques, we cite the seminal works \cite{Friedman73}, for the zero-sum case, and \cite{BensoussanFriedman77}, for the nonzero-sum case. For recent examples, see \cite{ChenDaiWan13,DeAnFerMor18}, and the references therein. It is also worth mentioning \cite{DeAnFer18}, which provides a link between two-player nonzero-sum games of optimal stopping and two-player nonzero-sum games of singular control. 

Regarding stochastic games with impulse controls, as well as games with mixed impulse and continuous controls, the existing literature almost entirely focuses on the zero-sum case. We refer to \cite{Azim17,Zhang11} for zero-sum games where one player uses continuous controls and the opponent takes impulse controls. For zero-sum games where both players use impulse controls, we cite \cite{Stettner82} for the case with delay and the recent works \cite{Cosso12,ElasriMazid18} for a viscosity approach. In particular, we notice that the system of QVIs proposed in \cite{Cosso12} for zero-sum impulse games can be obtained as a particular case of our framework, see Section \ref{sec:stochImpGame} below. 

Only the two papers \cite{ChangWangWu13,ChangWu15} deal with some nonzero-sum stochastic differential games with impulse controls and finite horizon, using an approach based on backward stochastic differential equations and the maximum principle. Notice that in those two papers the sequence of stopping times along which impulses can be applied is given, hence the players can choose only the size of the impulses. To the best of our knowledge, nonzero-sum impulse games are here considered for the first time in a general form. 
}

%
%

\vspace{0.4cm}

The outline of the paper is the following. In Section \ref{sec:stochImpGame} we rigorously formulate the general impulse game and give the notions of admissible strategies and of Nash equilibrium. Section \ref{sec:VerThm} provides the associated system of QVIs and the corresponding verification theorem. In Section \ref{sec:example} we analytically compute a family of Nash equilibria for a one-dimensional impulse game and provide two further examples with numerical solutions. Finally, Section \ref{sec:conclusion} concludes.

{\color{black} \paragraph{Acknowledgements.} The authors would like to thank J\'er\^ome Renault, Fabien Gensbittel and the anonymous referees for their valuable comments and suggestions. The authors also gratefully acknowledge funding from the SID research project ``New perspectives in stochastic methods for finance and energy markets" and from the Visiting Scientist program of the University of Padova, and from the Finance For Energy Market Research Centre (FiME) in Paris.}

\section{Nonzero-sum stochastic impulse games}
\label{sec:stochImpGame}

In this section we introduce a general class of two-player nonzero-sum stochastic differential games with impulse controls. 

Let ($\Omega$, $\mathcal{F}$, $\{\mathcal{F}_t\}_{t\geq 0}$, $\mathbb{P}$) be a filtered probability space whose filtration satisfies the usual conditions of right-continuity and $\mathbb P$-completeness, and let $\{W_t\}_{t \geq 0}$ be a $k$-dimensional $\{\mathcal{F}_t\}_{t \geq 0}$-adapted Brownian motion. For every $t \geq 0$ and ${\color{black} \zeta \in L^2 (\mathcal F_t)}$, we denote by $Y^{t,\zeta}= \{Y_s^{t,\zeta}\}_{s\ge t}$ a solution to the problem
\begin{equation}
\label{SDE}
d Y^{t,\zeta}_s = b(Y ^{t,\zeta}_s) ds + \sigma(Y ^{t,\zeta}_s) dW_s, \qquad s \geq t, 
\end{equation}
with initial condition $Y^{t,\zeta}_t=\zeta$, where $b: \rr^d \to \rr^d$ and $\sigma: \rr^d \to \rr^{d \times k}$ are given functions. Throughout the whole paper, we assume that the coefficients $b$ and $\sigma$ are globally Lipschitz continuous, i.e.~there exists a constant $K>0$ such that for all $y_1,y_2 \in \rr^d$ we have
\begin{equation*}
|b(y_1) - b(y_2)| + |\sigma(y_1)-\sigma(y_2)| \le K|y_1-y_2|,
\end{equation*}
so that \eqref{SDE} admits a unique strong solution satisfying classical a-priori estimates {\color{black}(see, e.g., \cite[Sect.~5.2]{Oksendal03} among others)}.

We consider two players, that will be indexed by $i \in \{1,2\}$. Let $S$ be an open subset of $\rr^d$ and let $Z_i$ be a fixed \textcolor{black}{non-empty} subset of $\rr^{l_i}$, with $l_i \in \nn$. Equation \eqref{SDE} models the underlying process when none of the players intervenes. If player $i$ intervenes with some impulse $\delta \in Z_i$, the process is shifted from its current state $x$ to a new state $\Gamma^i(x, \delta)$, where $\Gamma^i: S \times Z_i \to S$ is a given continuous function. Each intervention corresponds to a cost for the intervening player and to a gain for the opponent, both depending on the state $x$ and the impulse $\delta$. 

The action of the players is modelled via discrete-time controls: an impulse control for player $i$ is a sequence 
\begin{equation*}
\big\{ (\tau_{i,k},\delta_{i,k}) \big\}_{ k \geq 1},
\end{equation*}
where $\{ \tau_{i,k} \}_{k}$ is a non-decreasing sequence of stopping times (the intervention times) and $\{ \delta_{i,k} \}_{k}$ are $Z_i$-valued $\mathcal{F}_{\tau_{i,k}}$-measurable random variables (the corresponding impulses).

For the sake of tractability, we assume that the behaviour of the players, modelled by impulse controls, is driven by strategies, which are defined as follows.
\begin{defn}\label{defPhi}
A strategy for player $i\in \{1,2\}$ is a pair $\vphi_i =(\mathcal{C}_i,\xi_i)$, where $\mathcal{C}_i$ is a fixed open subset of $S$ and $\xi_i$ is a continuous function from $S$ to $Z_i$. 
\end{defn}

Strategies determine the action of the players in the following sense. Let $x \in S$ be an initial value for the state variable. Once some strategies $\vphi_i=(\mathcal{C}_i,\xi_i)$, $i \in \{1,2\}$, have been chosen, a pair of impulse controls $\{(\tau^{x;\vphi_1,\vphi_2}_{i,k}, \delta^{x;\vphi_1,\vphi_2}_{i,k})\}_{k \geq 1}$ is uniquely defined by the following procedure: 
\begin{equation}
\label{ideastrategies}
\begin{aligned}
&\text{- player $i$ intervenes if and only if the process exits from $\mathcal{C}_i$,} 
\\[-0.1cm]
&\text{\,\, in which case the impulse is given by $\xi_i(y)$, where $y$ is the state;}
\\[-0.1cm]
&\text{- if both players want to act, player 1 has the priority; 
}
\\[-0.1cm]
&\text{- the game ends when the process exits from $S$.}
\end{aligned}
\end{equation}
{\color{black} The second condition in \eqref{ideastrategies} solves the conflict situation when both players would like to shift the process, at the same time and to different states. In such cases, player 1 has the priority. Notice that this does not prevent player 2 from intervening immediately after player 1, if the new state is outside $\mathcal{C}_2$. However, we will forbid sequences of infinitely many simultaneous interventions. We will further comment on that later in this section, see Definition \ref{AdmStrat} and Remark \ref{RemSubsInterv}.}


In the following definition we provide a rigorous formalization of the controls associated to a pair of strategies and the corresponding controlled process, which we denote by $X^{x;\vphi_1,\vphi_2}$. Moreover $O$ denotes a generic subset of $S$ and, finally, we adopt the conventions $\inf \emptyset = \infty$ and $[\infty,\infty[=\emptyset$.
\begin{defn}
\label{controls}
Let $x \in S$ and let $\vphi_i =(\mathcal{C}_i,\xi_i)$ be a strategy for player $i\in \{1,2\}$. For $k \in \{0, \dots, \bar k\}$, where $\bar k = \sup \{k \in \nn \cup \{0\} : \widetilde \tau_k < \alpha^S_k \}$, we define by induction $\widetilde\tau_0=0$, $x_0=x$, $\widetilde X^0 = Y^{ \widetilde\tau_{0}, x_{0}}$, $\alpha_0^S =\infty$, and
\begin{align*}
& \alpha^O_k = \inf \{ s > \widetilde\tau_{k-1} : \widetilde X^{k-1}_s \notin O \},  && \text{{\footnotesize[exit time from $O \subseteq S$]}} \\	
&  \widetilde\tau_k = \alpha^{\mathcal{C}_1}_k \land \alpha^{\mathcal{C}_2}_k,  && \text{{\footnotesize[intervention time]}} \\
& m_k = \mathbbm{1}_{ \{ \alpha^{\mathcal{C}_1}_k \leq \alpha^{\mathcal{C}_2}_k \} } + 2 \, \mathbbm{1}_{ \{ \alpha^{\mathcal{C}_2}_k < \alpha^{\mathcal{C}_1}_k \} }, && \text{{\footnotesize[index of the player interv.~at $ \widetilde \tau_k$]}} \\
&  \widetilde\delta_k = \xi_{m_k} \big( \widetilde X^{k-1}_{ \widetilde \tau_k} \big)\mathbbm{1}_{ \{  \widetilde \tau_k < \infty \} }, && \text{{\footnotesize[impulse]}} \\
& x_k=\Gamma^{m_k} \big( \widetilde X^{k-1}_{ \widetilde\tau_k},  \widetilde\delta_k \big)\mathbbm{1}_{ \{  \widetilde \tau_k < \infty \} }, && \text{{\footnotesize[starting point for the next step]}} \\
& \widetilde X^k= \widetilde X^{k-1} \mathbbm{1}_{[0,  \widetilde \tau_k[} +  Y^{ \widetilde\tau_{k},x_{k}} \mathbbm{1}_{[ \widetilde \tau_k, \infty[}. && \text{{\footnotesize[contr.~process up to the $k$-th interv.]}}
\end{align*} 
Let $\bar k_i$ be the number of interventions by player $i \in \{1,2\}$ before the end of the game, and, in the case where $\bar k_i \neq 0$, let $\eta(i,k)$ be the index of her $k$-th intervention ($1 \leq k \leq \bar k_i$):
\begin{gather*}
\bar k_i = {\displaystyle\sum}_{1 \leq h \leq \bar k} \mathbbm{1}_{\{m_h=i\}},
\qquad\quad
\eta(i,k) = \min \Big\{ l \in \nn : {\displaystyle\sum}_{1 \leq h \leq l} \mathbbm{1}_{\{m_h=i\}} = k \Big\}.
\end{gather*}
{\color{black} Assume now that the times $\{\tilde \tau_k\}_{0\leq k \leq \bar k}$ never accumulate strictly before $\alpha^S_{\bar k}$, that is, we assume that in the event $\{\bar k=+\infty\}$ we have $\lim_{k \to \bar k}\tilde \tau_{k} = \alpha^S_{\bar k}$ (with the convention $\alpha^S_\infty= \sup_k \alpha^S_k$). The controlled process $X^{x;\vphi_1,\vphi_2}$ and the exit time $\tau^{x;\vphi_1,\vphi_2}_S$ are defined by (with the convention $\widetilde X^{\infty} = \lim_{k \to +\infty}\widetilde X^{k}$)}
\begin{equation*}
X^{x;\vphi_1,\vphi_2} := \widetilde X^{\bar k}, 
\qquad\qquad
\tau_S^{x;\vphi_1,\vphi_2} := {\color{black} \alpha^S_{\bar k} =} \inf \{ s \geq 0 : X^{x;\vphi_1,\vphi_2}_s \notin S \}.
\end{equation*}
Finally, the impulse controls $\{(\tau^{x;\vphi_1,\vphi_2}_{i,k}, \delta^{x;\vphi_1,\vphi_2}_{i,k})\}_{k \geq 1}$, with $i \in \{1, 2 \}$, are defined by
\begin{equation}
\label{times}
\tau^{x;\vphi_1,\vphi_2}_{i,k} :=  
\begin{cases}
\widetilde \tau_{\eta(i,k)}, & k \leq \bar k_i,
\\
\tau^{x;\vphi_1,\vphi_2}_S, & k > \bar k_i,
\end{cases} 
\qquad\qquad
\delta^{x;\vphi_1,\vphi_2}_{i,k} :=  
\begin{cases}
\widetilde \delta_{\eta(i,k)}, & k \leq \bar k_i,
\\
0, & k > \bar k_i.
\end{cases}  
\end{equation}
\end{defn} 


When the context is clear and in order to shorten the notations, we will simply write $X$, $\tau_S$, $\tau_{i,k}$, $\delta_{i,k}$. {\color{black} Definition \ref{controls} deserves some comments. First, notice that player 1 has the priority in the case where both players are willing to intervene, i.e., when $\alpha_k^{\mathcal{C}_1} = \alpha_k^{\mathcal{C}_2}$. Also, we underline that $\bar k, \bar k_1,\bar k_2$ are random variables, reflecting the fact that the number of interventions depends on $\omega \in \Omega$. In particular, notice that in the event $\{\bar k= \infty\}$ the convention $\widetilde X^\infty_t = \lim_k \widetilde X^k_t$ is well-posed for $t \in [0,\tau_S[$, since $\widetilde X^k \equiv \widetilde X^{k+1}$ in $[0,\tilde \tau_{k+1}[$ and $\lim_k \tilde \tau_k = \alpha^S_{\bar k} = \tau_S$ by assumption. Furthermore, we remark that if player $i \in \{1,2\}$ intervenes only a finite number of times, i.e., if $\bar k_i$ is finite, then the tail of the control is conventionally set to $(\tau_{i,k}, \delta_{i,k}) = (\tau_S, 0)$ for $k > \bar k_i$. However, notice that the controlled process $X$ does not jump at time $\tau_S$. Indeed, by definition the process only jumps at time $\tilde \tau_k$ with $k \leq \bar k$, in which case we have $\tilde \tau_k< \alpha^S_{k} \leq \alpha^S_{\bar k}=\tau_S$. We will further comment on the choice of such a tail later in this section, after Definition \ref{AdmStrat}. Finally, we remark that Definition \ref{controls} could be easily extended, even if at the cost of some additional technicalities, to the pathological case where the times $\{\tilde \tau_k\}_{0 \leq k \leq \bar k}$ accumulate before the end of the game. However, this is a degenerate situation which is not considered in this paper and, in general, in impulse control literature: see Definition \ref{AdmStrat} and the corresponding comments.}

In the following lemma we give a rigorous formulation to the properties outlined in \eqref{ideastrategies}.

\begin{lemma}
\label{lemmaprocess}
Let $x \in S$ and let $\vphi_i =(\mathcal{C}_i,\xi_i)$ be a strategy for player $i\in\{1,2\}$. 
\begin{itemize}
\item[-] The process $X$ admits the following representation (with the convention $[\infty,\! \infty[=\!\emptyset\!$):
\begin{equation}
\label{representation}
X_s = \sum_{k=0}^{\bar k -1} Y^{ \widetilde \tau_k,x_k}_s \mathbbm{1}_{ [ \widetilde \tau_k, \widetilde \tau_{k+1}[}(s) + Y^{ \widetilde \tau_{\bar k},x_{\bar k}}_s \mathbbm{1}_{ [ \widetilde \tau_{\bar k}, \infty[ }(s).
\end{equation}	
\item[-] The process $X$ is right-continuous. More precisely, $X$ is continuous and satisfies Equation \eqref{SDE} in $[0,\infty[ \, \setminus \, \{\tau_{i,k} : \tau_{i,k} < \tau_S \}$, whereas $X$ is discontinuous at $\{\tau_{i,k} : \tau_{i,k} < \tau_S \}$, where we have  
\begin{equation}
\label{prop1}
X_{\tau_{i,k}} = \Gamma^i \big(X_{(\tau_{i,k})^-}, \delta_{i,k} \big),
\qquad
\delta_{i,k} = \xi_i \big(X_{(\tau_{i,k})^-} \big),
\qquad
X_{(\tau_{i,k})^-} \in \partial \mathcal{C}_i.
\end{equation}
\item[-] The process $X$ never exits from the set $\mathcal{C}_1 \cap \mathcal{C}_2$.
\end{itemize}
\end{lemma}

\begin{proof}
We just prove the first property in \eqref{prop1}, the other ones being immediate. Let $i \in \{1,2\}$, $ k \geq 1$ with $\tau_{i,k}< \tau_S$ and set $\sigma = \eta(i,k)$, with $\eta$ as in Definition \ref{controls}. By \eqref{times}, \eqref{representation} and Definition \ref{controls}, we have
\begin{multline*}
X_{\tau_{i,k}} 
= X_{ \widetilde \tau_{\sigma}} 
= Y ^{ \widetilde\tau_\sigma, x_\sigma}_{ \widetilde\tau_\sigma} = x_{\sigma}
= \Gamma^i \big( \widetilde X^{\sigma-1}_{  \widetilde\tau_{\sigma} } ,  \widetilde\delta_\sigma \big) 
\\
= \Gamma^i \big( \widetilde X^{\sigma-1}_{  ( \widetilde\tau_{\sigma})^- } ,  \widetilde\delta_\sigma \big) 
= \Gamma^i \big( X_{  ( \widetilde\tau_{\sigma})^- } ,  \widetilde\delta_\sigma \big)
= \Gamma^i \big( X_{  (\tau_{i,k})^- } , \delta_{i,k} \big),
\end{multline*}
where in the fifth equality we have used the continuity of the process $\widetilde X ^{\sigma-1}$ in $[ \widetilde \tau_{\sigma-1}, \infty[$ and in the next-to-last equality we exploited the fact that $\widetilde X ^{\sigma-1} \equiv X$ in $[0,  \widetilde\tau_{\sigma}[$. 
\end{proof}

Each player aims at maximizing her payoff, consisting of four discounted terms: a running payoff, the costs due to her interventions, the gains due to her opponent's interventions and a terminal payoff. More precisely, for each $i \in \{1,2\}$ we consider $\rho_i > 0$ (the discount rate) and continuous functions $f_i: S \to \rr$ (the running payoffs), $h_i: \partial S \to \rr$ (the terminal payoffs) and $\phi_{i}: S \times Z_i \to \rr$, $\psi_{i}: S \times Z_j \to \rr$ (the interventions' costs and gains), where $j \in \{1,2\}$ with $j \neq i$. The payoff of player $i$ is defined as follows.

\begin{defn}
Let $x \in S$, let $(\vphi_1,\vphi_2)$ be a pair of strategies and let $\tau_S$ be defined as in Definition \ref{controls}. For each $i\in \{1,2\}$, provided that the right-hand side exists and is finite, we set 
\begin{multline}
\label{J}
J^i(x;\vphi_1,\vphi_2) := 
\eee \bigg[ \int_0^{\tau_S} e^{-\rho_i s} f_i(X_s) ds 
+ \sum_{ k \geq 1  \, : \, \tau_{i,k} < \tau_S } e^{-\rho_i \tau_{i,k}} \phi_{i} \Big( X_{(\tau_{i,k})^-}, \delta_{i,k} \Big)
\\
+ \sum_{ k \geq 1 \, : \, \tau_{j,k} < \tau_S } e^{-\rho_i \tau_{j,k}} \psi_{i} \Big( X_{(\tau_{j,k})^-}, \delta_{j,k} \Big)
+ e^{-\rho_i \tau_S} h_i (X_{\tau_S})\mathbbm{1}_{ \{ \tau_S < +\infty \} } \bigg],
\end{multline}
where $j\in \{1,2\}$ with $j \neq i$ and $\{(\tau_{i,k}, \delta_{i,k})\}_{k \geq 1 }$ is the impulse control of player $i$ associated to the strategies $\vphi_1,\vphi_2$. 
\end{defn}

As usual in control theory, the subscript in the expectation denotes conditioning with respect to the available information (hence, it recalls the starting point). Notice that in the summations above we do not consider stopping times which equal $\tau_S$, since the game ends in $\tau_S$. {\color{black} Moreover, by Definition \ref{controls} the process does not jump at time $\tau_S$, so that in \eqref{J} it is legitimate to write $h_i(X_{\tau_S})$ instead of $h_i(X_{(\tau_S)^-})$.}

In order for $J^i$ in \eqref{J} to be well defined, we now introduce the set of admissible strategies in $x \in S$.

\begin{defn}
\label{AdmStrat}
Let $x \in S$ and $\vphi_i = (\mathcal{C}_i,\xi_i)$ be a strategy for player $i  \in \{1,2\}$. We use the notations of Definition \ref{controls} and we say that the pair $(\vphi_1, \vphi_2)$ is $x$-admissible if: 
\begin{enumerate}
\item for $i \in \{1,2\}$, the following random variables are in $L^1(\Omega)$:
\begin{equation}
\label{L1}
\begin{gathered}
\int_0^{\tau_S} e^{-\rho_i s} |f_i|(X_s)ds, 
\qquad
e^{-\rho_i \tau_S}|h_i|(X_{\tau_S}),
\\
\sum_{ \tau_{i,k} < \tau_S } e^{-\rho_i \tau_{i,k}} |\phi_{i}| ( X_{(\tau_{i,k})^-}, \delta_{i,k} ),
\qquad
\sum_{ \tau_{i,k} < \tau_S } e^{-\rho_i \tau_{i,k}} |\psi_{i}| ( X_{(\tau_{i,k})^-}, \delta_{i,k} );
\end{gathered}
\end{equation}
\item for each $p \in \nn$, the random variable $\|X\|_\infty = \sup_{t\geq0}|X_t|$ is in $L^p(\Omega)$:
\begin{equation}
\label{L1process}
\eee[\|X\|^p_\infty] < \infty;
\end{equation}
\item for $i \in \{1,2\}$, we have
\begin{equation}
\label{LimImp}
\lim_{k \rightarrow + \infty} \tau_{i,k} = \tau_S.
\end{equation}
\end{enumerate}
We denote by $\Phi_x$ the set of the $x$-admissible pairs. 
\end{defn}
Thanks to the first condition in Definition \ref{AdmStrat}, the payoffs $J^i(x;\vphi_1,\vphi_2)$ are well-defined. The second condition will be used in the proof of the Verification Theorem \ref{thm:verification}. {\color{black} We observe that \eqref{L1process} is practically reasonable. Indeed, in practical applications of competitive games the mutual continuation region is usually a bounded set (e.g., see the nonzero-sum problems in Section \ref{sec:example}, or the stopping games in \cite{DeAnFerMor18}), so that $|X|$ is bounded and \eqref{L1process} holds. The third condition prevents each player from accumulating the interventions before $\tau_S$. This is a usual assumption in impulse control theory, see, e.g., \cite[Ch.~6]{OksendalSulem07}. In particular, by this condition we do not admit a sequence of infinitely many interventions in a same instant $t < \tau_S$. However, a finite number of simultaneous interventions is allowed at any time, as we detail in Remark \ref{RemSubsInterv}. Finally, notice that the third relation is always true if player $i$ intervenes a finite number of times, since in Definition \ref{controls} the tail of the controls was conventionally set to be $(\tau_{i,k},\delta_{i,k})=(\tau_S,0)$ for $k >\bar k_i$.} 

We conclude the section with the definition of Nash equilibrium and the corresponding payoff functions in our setting. \textcolor{black}{Notice that the functions $V_i$ are not uniquely defined, but depend on the Nash equilibrium considered.}
\begin{defn}
\label{DefNash}
Given $x \in S$, we say that $(\vphi_1^*,\vphi_2^*) \in \Phi_x$ is a Nash equilibrium of the game if
\begin{gather*}
J^1(x;\vphi_1^*,\vphi_2^*) \geq J^1(x;\vphi_1,\vphi_2^*), 
\qquad \forall \vphi_1 \,\, \text{s.t.} \,\, (\vphi_1,\vphi_2^*)\in\Phi_x, \\
J^2(x;\vphi_1^*,\vphi_2^*) \geq J^2(x;\vphi_1^*,\vphi_2), 
\qquad \forall \vphi_2 \,\,\text{s.t.} \,\, (\vphi_1^*,\vphi_2)\in\Phi_x. 
\end{gather*}
Finally, the payoff functions of a Nash equilibrium are defined as follows: if $x \in S$ and a Nash equilibrium $(\vphi_1^*,\vphi_2^*) \in \Phi_x$ exists, we set for $i \in \{1,2\}$
$$
V_i(x) := J^i(x;\vphi_1^*,\vphi_2^*).
$$
\end{defn}

\begin{remark}
\label{RemBdr}
\textcolor{black}{\emph{It is technically convenient to deal with functions $V_i$ defined on a closed set. For this reason, we extend the definitions of $J^i$ and $V_i$ to include the trivial cases $x \in \partial S$. Since $S$ is an open set, in this case the game immediately stops, so that we have $\tau_S=0$ and $J^i(x;\vphi_1,\vphi_2)=h_i(x)$ for each $(\vphi_1,\vphi_2)$, and hence $V_i(x)=h_i(x)$.}}
\end{remark}

\begin{remark}
\label{RemSubsInterv}
\textcolor{black}{\emph{We remark that a finite number of subsequent simultaneous interventions is allowed in our framework. Namely, if one of the players intervenes in $t \geq 0$, there could be another intervention immediately afterwards, that is, still at time $t$. If the new state is outside the continuation region of one of the players, we could have a third intervention in that same instant $t$, and so on. However, this can only happen finitely many times, since it is clear that an infinite sequence, here prevented by the third condition in Definition \ref{AdmStrat}, would correspond to a degenerate game.}}
\end{remark}

\begin{remark}
	\label{RemCommentsHyp} 
	\textcolor{black}{\emph{Our definition of strategies is different from the one used in \cite{Cosso12}, which is an adaptation of the notion of non-anticipative strategies \`a la Elliott-Kalton \cite{ElliottKalton72} (see \cite{FlemSoug89} for the stochastic case). In the latter, the strategy of any player is designed as a response to her competitor's strategy in a non-anticipative way, i.e., it cannot depend on the future strategy. This notion of strategy has been successfully applied to the study of zero-sum stochastic differential games with continuous controls, and further strengthened in \cite{BuckCardRain04} to deal with their nonzero-sum counterparts. It turns out that the notion of non-anticipative strategy is very well suited for the viscosity solutions approach. In the present paper, we have decided to focus on feedback strategies of a particular form mainly for tractability purposes. Indeed, they are a generalisation of threshold type strategies, that revealed very effective in order to obtain explicit equilibria (see, e.g., \cite{DeAnFerMor18}). Moreover, each player's actions are still responses to her competitor's actions even if only through the controlled state variable.}}
	
\end{remark}

\section{A verification theorem}
\label{sec:VerThm}

In this section we define a suitable differential problem for the payoff functions at some Nash equilibrium in nonzero-sum impulse games (see Section \ref{ssec:QVI}) and prove a verification theorem for such games (see Section \ref{ssec:Verif}).

\subsection{The quasi-variational inequality problem}
\label{ssec:QVI}

We now introduce the differential problem that should be satisfied by the payoff functions of some Nash equilibrium in our games: this will be key for stating the verification theorem in the next section.

Let us consider an impulse game as in Section \ref{sec:stochImpGame}. Assume that at some Nash equilibrium the payoff functions $V_1,V_2$ are defined for each $x \in \overline S$ \textcolor{black}{(compare Remark \ref{RemBdr}, while for all the operators introduced in this section it is enough to consider the open set $S$). Also assume that} for $i \in\{1,2\}$ there exists a {\color{black} unique, finite, measurable} function $\delta_i$ from $S$ to $Z_i$ such that 
\begin{equation}
\label{def:deltai}
\{\delta_i(x)\} = \argmax_{\delta \in Z_i} \big\{ V_i(\Gamma^i(x,\delta))+\phi_i(x,\delta)\big\}, 
\end{equation}
for each $x \in S$. We define the four intervention operators by
\begin{equation}
\label{MH}
\begin{gathered}
\mm_i V_i(x) = V_i\big(\Gamma^i(x,\delta_i(x))\big) + \phi_i \big(x,\delta_i(x)\big), 
\\
\hh_i V_i(x) = V_i\big(\Gamma^j(x,\delta_j(x))\big) + \psi_i \big(x,\delta_j(x)\big), 
\end{gathered}
\end{equation}
for $x \in S$ and $i,j \!\in\!\{1,2\}$, with $i \!\neq\! j$. Notice that $\mm_i V_i (\cdot) \!=\! \max_\delta \{ V_i ( \Gamma^i (\cdot,\delta) ) + \phi_i (\cdot,\delta) \}$.

The functions in \eqref{def:deltai} and \eqref{MH} have an immediate and intuitive interpretation. Let $x$ be the current state of the process; if player $i$ (resp.~player $j$) intervenes with impulse $\delta$, the present equilibrium payoff for player $i$ can be written as $V_i(\Gamma^i(x,\delta)) + \phi_i(x,\delta)$ (resp.~$V_i(\Gamma^j(x,\delta)) + \psi_i(x,\delta)$): we have considered  the payoff in the new state and the intervention cost (resp.~gain). Hence, $\delta_i(x)$ in \eqref{def:deltai} is the impulse that player $i$ would use in case she wants to  intervene.

Similarly, $\mm_iV_i(x)$ (resp.~$\hh_iV_i(x)$) represents the payoff for player $i$ when player $i$ (resp.~player $j \neq i$) takes the best immediate action and behaves optimally afterwards. Notice that it is not always optimal to intervene, so $\mm_iV_i(x) \le V_i(x)$, for each $x \in S$, and that player $i$ should intervene (with impulse $\delta_i(x)$) only if $\mm_iV_i(x) = V_i(x)$. This gives a heuristic formulation of Nash equilibria, provided that an explicit expression for $V_i$ is available. The verification theorem will give a rigorous proof to this heuristic argument. We now characterize the payoff functions $V_i$.

Assume $V_1,V_2 \in C^2(\overline S)$ (weaker conditions will be given later) and define 
\begin{equation*}
\aaa V_i = b \cdot \nabla V_i + \frac{1}{2} \mbox{tr} \left( \sigma \sigma^t D^2 V_i\right), 
\end{equation*}
where $b, \sigma$ are as in \eqref{SDE}, $\sigma^t$ denotes the transpose of $\sigma$ and $\nabla V_i, D^2 V_i$ are the gradient and the Hessian matrix of $V_i$, respectively. We are interested in the following quasi-variational inequalities (QVIs) for $V_1,V_2$, where $i,j \in\{1,2\}$ and $i \neq j$:
\begin{subequations}
\label{pbNEW}
\begin{align}
& V_i = h_i,  && \text{in} \,\,\, \partial S, \label{pbNEW-dS} \\
& \mm_jV_j - V_j \leq 0, && \text{in} \,\,\, S, \label{pbNEW-S}  \\
& \hh_iV_i-V_i=0, && \text{in} \,\,\, \{\mm_jV_j - V_j = 0\}, \label{pbNEW-dj}  \\
& \max\big\{\aaa V_i -\rho_i V_i + f_i, \mm_iV_i-V_i \}=0, && \text{in} \,\,\, \{\mm_jV_j - V_j < 0\}  \label{pbNEW-j}  .
\end{align}
\end{subequations}

We now provide some intuition behind conditions \eqref{pbNEW-dS}-\eqref{pbNEW-j}.
First of all, the terminal condition is obvious. Moreover, as we already noticed,  \eqref{pbNEW-S}  is a standard condition in impulse control theory. For (\ref{pbNEW-dj}), if player $j$ intervenes (i.e.,~$\mm_jV_j - V_j =0$), by the definition of Nash equilibrium we expect that player $i$ does not lose anything: this is equivalent to $\hh_iV_i - V_i =0$, otherwise it would be in her interest to deviate. On the contrary, if player $j$ does not intervene (hence $\mm_jV_j - V_j <0$), then the problem for player $i$ becomes a classical one-player impulse control one, $V_i$ satisfies $\max\big\{\aaa V_i -\rho_i V_i + f_i, \mm_iV_i-V_i \}=0$. In short, the latter condition says that $ \aaa V_i-\rho_i V_i + f_i \leq 0$, with equality in case of non-intervention (i.e.,~$\mm_iV_i - V_i <0$). 


\begin{remark}
\label{RemC2}
\emph{We notice that $\aaa V_i$ only appears in $\{\mm_jV_j - V_j < 0\}$, so that  $V_i$ needs to be of class $C^2$ only in such a region (indeed, this assumption can be slightly relaxed, as we will see). This represents a difference to the one-player case, where the value function is usually required to be twice differentiable almost everywhere in $S$, see \cite[Thm.~6.2]{OksendalSulem07}.}
\end{remark}

\paragraph{The zero-sum case.} A verification theorem will be provided in the next section. Here, as a preliminary check, we show that we are indeed generalizing the system of QVIs provided in \cite{Cosso12}, where the zero-sum case is considered. We show that, if we assume
\begin{equation}
\label{zerosum}
\begin{aligned}
&f:=f_1=-f_2, &&\qquad \phi:=\phi_1=-\psi_2, &&\qquad \psi:=\psi_1=-\phi_2, \\
&h:=h_1=-h_2, &&\qquad Z := Z_1 = Z_2, &&\qquad \Gamma:= \Gamma^1 = \Gamma^2, 
\end{aligned}
\end{equation}
so that $V:=V_1=-V_2$, then the problem in \eqref{pbNEW} reduces to the one considered in \cite{Cosso12}. To shorten the equations, we assume $\rho_1=\rho_2=0$ (this makes sense since in \cite{Cosso12} a finite-horizon problem is considered). First of all, we define 
\begin{gather*}
\widetilde \mm V (x) := \sup_{\delta \in Z} \big\{ V(\Gamma(x,\delta))+\phi(x,\delta) \big\},\\
\widetilde \hh V (x) := \inf_{\delta \in Z} \big\{ V(\Gamma(x,\delta))+\psi(x,\delta)\big\},
\end{gather*}
for each $x \in S$. It is easy to see that, under the conditions in \eqref{zerosum}, we have
\begin{equation}
\label{subst}
\mm_1 V_1 = \widetilde \mm V, \qquad \mm_2 V_2 = - \widetilde \hh V, \qquad \hh_1 V_1 = \widetilde \hh V, \qquad \hh_2 V_2 = - \widetilde \mm V.
\end{equation}
By using \eqref{subst}, problem \eqref{pbNEW} becomes
\begin{subequations}
\label{pbTEMP}
\begin{align}
& V = h,  && \text{in} \,\,\, \partial S,  \label{pbTEMP-dS}  \\
& \widetilde \mm V \leq V \leq \widetilde \hh V, && \text{in} \,\,\, S,  \label{pbTEMP-S}  \\
& \aaa V + f \leq 0, && \text{in} \,\,\, \{ V = \widetilde \mm V \},  \label{pbTEMP-M}  \\
& \aaa V + f = 0, && \text{in} \,\,\, \{ \widetilde \mm V < V < \widetilde \hh V\},  \label{pbTEMP-MH} \\
& \aaa V + f \geq 0, && \text{in} \,\,\, \{ V = \widetilde \hh V \} \label{pbTEMP-H}  .
\end{align}
\end{subequations}
Simple computations show that problem \eqref{pbTEMP} is equivalent to
\begin{subequations}
\label{pbCOSSO}
\begin{align}
& V = h,  && \text{in} \,\,\, \partial S,  \label{pbCOSSO-dS}  \\
& \widetilde \mm V -V \leq 0, && \text{in} \,\,\, S,  \label{pbCOSSO-S}  \\
& \min \{ \max \{ \aaa V + f, \widetilde \mm V - V \}, \widetilde \hh V - V \}= 0, && \text{in} \,\,\, S,  \label{pbCOSSO-qvi} 
\end{align}
\end{subequations}
which is exactly the problem studied in \cite{Cosso12}, as anticipated. 

\begin{lemma}
\label{lemmaZEROSUM}
Problems \eqref{pbTEMP} and \eqref{pbCOSSO} are equivalent.
\end{lemma}

\begin{proof}
\textcolor{black}{Postponed to Appendix \ref{SecApp}.}
\end{proof}

\subsection{Statement and proof}
\label{ssec:Verif}

We provide here the main mathematical contribution of this paper, which is a verification theorem for the problems formalized in Section \ref{sec:stochImpGame}. 

\begin{thm}[\textbf{Verification theorem}] 
\label{thm:verification}
Let all the notations and working assumptions in Section \ref{sec:stochImpGame} be in force and let $V_i$ be a function from $\overline S$ to $\rr$, with $i \in \{1,2\}$. Assume that \eqref{def:deltai} holds and set $\mathcal D_i \!:=\! \{ \mm_i V_i - V_i < 0 \}$, with $\mm_i V_i$ as in \eqref{MH}. Moreover, for $i \!\in\! \{1,2\}$ assume that:
\begin{itemize}
\item[(i)] $V_i$ is a solution to  \eqref{pbNEW-dS}-\eqref{pbNEW-j};
\item[(ii)] $V_i \in C^2(\mathcal D_j \setminus \partial \mathcal D_i) \cap C^1(\mathcal D_j) \cap C(\overline S)$ and it has polynomial growth;
\item[(iii)] $\partial \mathcal D_i $ is a Lipschitz surface (i.e.~it is locally the graph of a Lipschitz function), and $V_i$ has locally bounded derivatives up to the second order in some neighbourhood of $\partial \mathcal D_i$.
\end{itemize}
Finally, let $x \in S$ and assume that $(\vphi_1^*,\vphi_2^*) \in \Phi_x$, with 
\begin{equation*}
\vphi_i^*=(\mathcal D_i, \delta_i),
\end{equation*}
where $i \in \{1,2\}$, the set $\mathcal D_i$ is as above and the function $\delta_i$ is as in \eqref{def:deltai}. Then, 
\begin{equation*}
\mbox{$(\vphi_1^*,\vphi_2^*)$ is a Nash equilibrium and $V_i(x)=J^i(x;\vphi_1^*,\vphi_2^*)$ for $i \in \{1,2\}$.}
\end{equation*} 	
\end{thm}

\begin{remark}
\emph{Practically, the Nash strategy is characterized as follows: player $i$ intervenes only if the controlled process exits from the region $\{\mm_i V_i - V_i <0\}$ (equivalently, only if $\mm_i V_i(x) = V_i(x)$, where $x$ is the current state). When this happens, her impulse is $\delta_i(x)$.}
\end{remark}

\begin{remark}
{\color{black} \emph{Some technical steps and approximations in the proof of Theorem \ref{thm:verification} may hide the idea on which this result is based: if $\vphi_1$ is a strategy for player 1 such that $(\vphi_1,\vphi_2^*) \in \Phi_x$, by the It\^o formula (here, just heuristically) and the four conditions in the QVI problem \eqref{pbNEW}, we get
\begin{align*}
V_1(x) \,\, \text{``}\!&=\!\!\text{''} \,\,\, \eee \bigg[ \!\!-\!\! \int_{0}^{\tau_S} \!\! e^{-\rho_1 s} (\aaa V_1 \!-\! \rho_1V_1) (X_s) ds - \!\!\!\! \sum_{\tau_{1,k}< \tau_S} \!\!\! e^{-\rho_1 \tau_{1,k}} \Big( V_1 \big( X_{\tau_{1,k}} \big) \!-\! V_1 \big( X_{ ( \tau_{1,k} )^-} \big) \Big)
\\
& \hspace{1.5cm} - \sum_{\tau_{2,k}< \tau_S} e^{-\rho_1 \tau_{2,k}} \Big( V_1 \big( X_{\tau_{2,k}} \big) - V_1 \big( X_{ ( \tau_{2,k} )^-} \big) \Big) + e^{-\rho_1 \tau_S}V_1(X_{\tau_S})\mathbbm{1}_{ \{ \tau_{S} < +\infty \} } \bigg]\\
& \geq \eee \bigg[ \int_{0}^{\tau_{S}} e^{-\rho_1 s} f_1(X_s) ds + \sum_{\tau_{1,k}< \tau_{S}} e^{-\rho_1 \tau_{1,k}} \phi_1 \big( X_{ ( \tau_{1,k} )^-} , \delta_{1,k} \big)
\\
& \hspace{1.5cm} + \sum_{\tau_{2,k}< \tau_{S}} e^{-\rho_1 \tau_{2,k}} \psi_1 \big( X_{ ( \tau_{2,k} )^-} , \delta_{2,k} \big) + e^{-\rho_1 \tau_{S}} h_1(X_{\tau_S})\mathbbm{1}_{ \{ \tau_{S} < +\infty \} } \bigg] = J^1(x;\vphi_1,\vphi_2^*).
\end{align*}
When considering $\vphi=\vphi_1^*$, we get an equality by the definition of $\vphi_1^*$, so that $(\vphi_1^*,\vphi_2^*)$ is Nash equilibrium. However, several approximating sequences have to be considered, since we cannot directly apply the It\^o formula ($V_1$ is not regular enough) and take expectations (we are dealing with potentially infinite sums).}}
\end{remark}

\begin{remark}	
\emph{{\textcolor{black}{Notice that, since $V_i \in C(\overline S)$ for $i \in \{1,2\}$, the sets $\mathcal{D}_i$ are open and the functions $\delta_i$ are measurable by the measurable maximum theorem in \cite[Thm.~18.19]{AlipBorder}.}} We also observe that, for the (candidate) equilibrium strategies in the theorem above, the properties in Lemma \ref{lemmaprocess} imply what follows (the notation is heavy, but it will be crucial to understand the proof of the theorem):
\begin{subequations}
\label{prop2}
{\allowdisplaybreaks
\begin{align}
& (\mm_1 V_1 - V_1) \big( X^{x;\vphi^*_1,\vphi_2}_s \big)  < 0 , 
\label{prop2-M1} \\
& (\mm_2 V_2 - V_2) \big( X^{x;\vphi_1,\vphi^*_2}_s \big)  < 0, 
\label{prop2-M2} \\
& \delta_{1,k}^{x;\vphi^*_1,\vphi_2}  = \delta_1 \bigg( X^{x;\vphi^*_1,\vphi_2}_{ \big(\tau_{1,k}^{x;\vphi^*_1,\vphi_2} \big)^-} \bigg),  \label{prop2-d1} \\
& \delta_{2,k}^{x;\vphi_1,\vphi^*_2}  = \delta_2 \bigg( X^{x;\vphi_1,\vphi^*_2}_{ \big(\tau_{2,k}^{x;\vphi_1,\vphi^*_2} \big)^-} \bigg),  \label{prop2-d2}  \\
& (\mm_1 V_1 - V_1) \bigg( X^{x;\vphi^*_1,\vphi_2}_{ \big( \tau_{1,k}^{x;\vphi^*_1,\vphi_2} \big)^-} \bigg)  =0,   \label{prop2-M1tau} \\
& (\mm_2 V_2 - V_2) \bigg( X^{x;\vphi_1,\vphi^*_2}_{ \big( \tau_{2,k}^{x;\vphi_1,\vphi^*_2} \big)^-} \bigg)  =0,  \label{prop2-M2tau}
\end{align}
}
\end{subequations}
\!\!\!\!\!\! for every strategies $\vphi_1, \vphi_2$ such that $(\vphi_1,\vphi_2^*), (\vphi_1^*,\vphi_2) \in \Phi_x$, every $s \geq 0$ and every $\tau_{i,k}^{x;\vphi_1,\vphi^*_2}$, $\tau_{i,k}^{x;\vphi_1^*,\vphi_2} < \infty$. }
\end{remark}

\begin{proof}
By Definition \ref{DefNash}, we have to prove that
\begin{equation*}
V_i(x)= J^i(x;\vphi_1^*,\vphi_2^*), 
\qquad V_1(x) \geq J^1(x;\vphi_1,\vphi_2^*), 
\qquad	V_2(x) \geq J^2(x;\vphi_1^*,\vphi_2), 
\end{equation*}
for every $i \in \{1,2\}$ and $(\vphi_1, \vphi_2)$ strategies such that $(\vphi_1,\vphi_2^*)\in\Phi_x$ and $(\vphi_1^*,\vphi_2)\in\Phi_x$. We show the results for $V_1$ and $J^1$, the arguments for $V_2$ and $J^2$ being symmetric.
	
\textit{Step 1: $V_1(x) \geq J^1(x;\vphi_1,\vphi_2^*)$.} Let $\vphi_1$ be a strategy for player 1 such that $(\vphi_1,\vphi_2^*) \in \Phi_x$. Here we will use the following shortened notation:
\begin{equation*}
X=X^{x;\vphi_1,\vphi^*_2}, \qquad 
\tau_{i,k}=\tau^{x;\vphi_1,\vphi^*_2}_{i,k}, \qquad 
\delta_{i,k}=\delta^{x;\vphi_1,\vphi^*_2}_{i,k}.
\end{equation*}	
{\color{black} In order to use the It\^o formula, we first need to approximate $V_1$ with regular functions. Since (ii) and (iii) hold, by \cite[proof of Thm.~10.4.1 and App.~D]{Oksendal03} there exists a sequence of functions $\{V_{1,j}\}_{j \in \nn}$ such that:
\begin{itemize}
	\item[(a)] $V_{1,j} \in C^2(\mathcal{D}_2) \cap C^0(\overline S)$, for each $j \in \nn$ (in particular, $\aaa V_{1,j}$ is well-defined in $\mathcal{D}_2$);
	\item[(b)] $V_{1,j} \to V_1$ as $j \to \infty$, uniformly on the compact subsets of $\overline S$;
	\item[(c)] $\{\aaa V_{1,j}\}_{j \in \nn}$ is locally bounded in $\mathcal{D}_2$ and $\aaa V_{1,j} \to \aaa V_1$ as $j \to \infty$, uniformly on the compact subsets of $\mathcal{D}_2 \setminus \partial \mathcal{D}_1$.
\end{itemize}
}
 {\color{black}  For each $r >0$ and $\ell \in \nn$, we set
\begin{equation}
\label{stopping}
\tau_{r,\ell} = \tau_r \land \tau_{1,\ell} \land \tau_{2,\ell},
\end{equation}
where $\tau_r = \inf \{ s > 0 : X_s \notin B(0,r) \} $ is the exit time from the ball with radius $r$. By \eqref{prop2-M2} we have that $X_s \in \mathcal{D}_2$ for each $s > 0$. Since $V_{1,j} \in C^2(\mathcal{D}_2)$ by (a), for each $j \in \nn$ we can apply It\^o's formula to the process $e^{-\rho_1 t}V_{1,j}(X_t)$ over the interval $[0, \tau_{r,\ell}[$, and take the conditional expectations: we get
\begin{multline}
\label{itoAPPROX}		
V_{1,j}(x) \!=\! \eee \bigg[ \!\!-\!\! \int_{0}^{\tau_{r,\ell}} \!\! e^{-\rho_1 s} (\aaa V_{1,j} \!-\! \rho_1 V_{1,j}) (X_s) ds - \!\!\!\! \sum_{\tau_{1,k}< \tau_{r,\ell}} \!\!\! e^{-\rho_1 \tau_{1,k}} \Big( V_{1,j} \big( X_{\tau_{1,k}} \big) \!-\! V_{1,j} \big( X_{ ( \tau_{1,k} )^-} \big) \Big)
\\
- \sum_{\tau_{2,k}< \tau_{r,\ell}} e^{-\rho_1 \tau_{2,k}} \Big( V_{1,j} \big( X_{\tau_{2,k}} \big) - V_{1,j} \big( X_{ ( \tau_{2,k} )^-} \big) \Big) + e^{-\rho_1 \tau_{r,\ell}}V_{1,j}\big(X_{(\tau_{r,\ell})^-}\big) \bigg].
\end{multline}	
Notice that \eqref{itoAPPROX} is well-defined by \eqref{stopping}: indeed, since $\tau_{r,\ell} \leq \tau_r$, $X$ belongs to the compact set $\overline{B(0,r)}$, where the continuous function $V_{1,j}$ is bounded; moreover, the two summations consist in a finite number of terms since $\tau_{r,\ell} \leq \tau_{1,\ell} \land \tau_{2,\ell}$. Also, notice that in \eqref{itoAPPROX} we need to write $V_{1,j}(X_{(\tau_{r,\ell})^-})$, due to the jump at time $\tau_{r,\ell}$. We now pass to the limit in \eqref{itoAPPROX} as $j \to \infty$: since $X$ belongs to the compact set $\overline{B(0,r)}$, by the uniform convergence in (b) and (c) we get}
\begin{multline}
\label{ito}	
V_1(x) \!=\! \eee \bigg[ \!\!-\!\! \int_{0}^{\tau_{r,\ell}} \!\! e^{-\rho_1 s} (\aaa V_1 \!-\! \rho_1V_1) (X_s) ds - \!\!\!\! \sum_{\tau_{1,k}< \tau_{r,\ell}} \!\!\! e^{-\rho_1 \tau_{1,k}} \Big( V_1 \big( X_{\tau_{1,k}} \big) \!-\! V_1 \big( X_{ ( \tau_{1,k} )^-} \big) \Big)
\\
- \sum_{\tau_{2,k}< \tau_{r,\ell}} e^{-\rho_1 \tau_{2,k}} \Big( V_1 \big( X_{\tau_{2,k}} \big) - V_1 \big( X_{ ( \tau_{2,k} )^-} \big) \Big) + e^{-\rho_1 \tau_{r,\ell}}V_1\big(X_{(\tau_{r,\ell})^-}\big) \bigg].
\end{multline}	

We now estimate each term in the right-hand side of \eqref{ito}. As for the first term, since $(\mm_2V_2 - V_2)(X_s) < 0$ by (\ref{prop2-M2}), from (\ref{pbNEW-j}) it follows that 
\begin{equation}
\label{magg1}
(\aaa V_1-\rho_1V_1) (X_s) \leq -f_1(X_s),
\end{equation}
for all $s \in [0,\tau_S]$. For the second term, {\color{black} let us consider any $k \in \nn$ and $\omega \in \Omega$ with $\tau_{1,k}(\omega)<\tau_S(\omega)$.} By (\ref{pbNEW-S}) and the definition of $\mm_1 V_1$ in \eqref{MH},we have 
\begin{align}
\label{magg2}
V_1 \big( X_{ ( \tau_{1,k} )^-} \big) & \geq \mm_1 V_1 \big( X_{ ( \tau_{1,k} )^-} \big)  \nonumber \\
& = \sup_{\delta \in Z_1} \big\{ V_1 \big( \Gamma^1 \big( X_{ ( \tau_{1,k} )^-} , \delta \big) \big) + \phi_1 \big( X_{ ( \tau_{1,k} )^-} , \delta \big) \big\} \nonumber \\
& \geq V_1 \big( \Gamma^1 \big( X_{ ( \tau_{1,k} )^-} , \delta_{1,k} \big) \big) + \phi_1 \big( X_{ ( \tau_{1,k} )^-} , \delta_{1,k} \big) \nonumber \\
& = V_1 \big( X_{ \tau_{1,k} } \big) + \phi_1 \big( X_{ ( \tau_{1,k} )^-} , \delta_{1,k} \big).
\end{align}
As for the third term, {\color{black} let us consider any $k \in \nn$ and $\omega \in \Omega$ with $\tau_{2,k}(\omega)<\tau_S(\omega)$.} By (\ref{prop2}f) we have $(\mm_2 V_2 - V_2) \big( X_{ ( \tau_{2,k} )^-} \big) = 0$; hence, the condition in (\ref{pbNEW-dj}), the definition of $\hh_1 V_1$ in \eqref{MH} and the expression of $\delta_{2,k}$ in (\ref{prop2-d2}) imply that
\begin{align}
\label{magg3}
V_1 \big( X_{ ( \tau_{2,k} )^-} \big) & = \hh_1 V_1 \big( X_{ ( \tau_{2,k} )^-} \big) \nonumber \\ 
& = V_1 \big( \Gamma^2 \big( X_{ ( \tau_{2,k} )^-} , \delta_{2} \big( X_{ ( \tau_{2,k} )^-}) \big) \big) + \psi_1 \big( X_{ ( \tau_{2,k} )^-} , \delta_{2} \big( X_{ ( \tau_{2,k} )^-})\big) \nonumber \\
& = V_1 \big( \Gamma^2 \big( X_{ ( \tau_{2,k} )^-} , \delta_{2,k} \big) \big) + \psi_1 \big( X_{ ( \tau_{2,k} )^-} , \delta_{2,k} \big) \nonumber \\
& = V_1 \big( X_{  \tau_{2,k} } \big) + \psi_1 \big( X_{ ( \tau_{2,k} )^-} , \delta_{2,k} \big). 
\end{align}
By \eqref{ito} and the estimates in \eqref{magg1}-\eqref{magg3} it follows that
\begin{multline*}
V_1(x) \geq \eee \bigg[ \int_{0}^{\tau_{r,\ell}} e^{-\rho_1 s} f_1(X_s) ds + \sum_{\tau_{1,k}< \tau_{r,\ell}} e^{-\rho_1 \tau_{1,k}} \phi_1 \big( X_{ ( \tau_{1,k} )^-} , \delta_{1,k} \big)
\\
+ \sum_{\tau_{2,k}< \tau_{r,\ell}} e^{-\rho_1 \tau_{2,k}} \psi_1 \big( X_{ ( \tau_{2,k} )^-} , \delta_{2,k} \big) + e^{-\rho_1 \tau_{r,\ell}} V_1 \big(X_{(\tau_{r,\ell})^-}\big) \bigg]. 
\end{multline*}

Thanks to the conditions in \eqref{L1},\eqref{L1process} and the polynomial growth of $V_1$ in (ii), we now use the dominated convergence theorem and pass to the limit, first as $r \to \infty$ and then as $\ell \to \infty$, {\color{black} so that the stopping times $\tau_{r,\ell}$ converge to $\tau_S$ by \eqref{LimImp}.} In particular, for the fourth term we notice that {\color{black} by (ii) and \eqref{L1process} we have }
\begin{equation}
\label{domcon}
V_1(X_{(\tau_{r,\ell})^-}) \leq C(1+|X_{(\tau_{r,\ell})^-}|^p) \leq C(1+\|X\|_\infty^p) \in L^1(\Omega),
\end{equation}
for suitable constants $C>0$ and $p \in \nn$; the corresponding limit immediately follows by the continuity of $V_1$ in the case $\tau_S < \infty$ and by \eqref{domcon} itself in the case $\tau_S = \infty$ (as a direct consequence of \eqref{L1process}, we have $\|X\|^p_\infty < \infty$ a.s.). Hence, we finally get
\begin{multline*}
V_1(x) \geq \eee \bigg[ \int_{0}^{\tau_{S}} e^{-\rho_1 s} f_1(X_s) ds + \sum_{\tau_{1,k}< \tau_{S}} e^{-\rho_1 \tau_{1,k}} \phi_1 \big( X_{ ( \tau_{1,k} )^-} , \delta_{1,k} \big)
\\
+ \sum_{\tau_{2,k}< \tau_{S}} e^{-\rho_1 \tau_{2,k}} \psi_1 \big( X_{ ( \tau_{2,k} )^-} , \delta_{2,k} \big) + e^{-\rho_1 \tau_{S}} h_1(X_{\tau_S})\mathbbm{1}_{ \{ \tau_{S} < +\infty \} } \bigg] = J^1(x;\vphi_1,\vphi_2^*).
\end{multline*}
	
\textit{Step 2: $V_1(x) = J^1(x;\vphi_1^*,\vphi_2^*)$.} We argue as in Step 1, but here all the inequalities are equalities by the properties of $\vphi^*_1$.
\end{proof}

As already noticed in Remark \ref{RemC2}, we stress that, unlike one-player impulse control problems, in our verification theorem the candidates are not required to be twice differentiable everywhere. For example, consider the case of player 1: as in the proof we always consider pairs of strategies in the form $(\vphi_1,\vphi_2^*)$, by \eqref{prop2-M2} the controlled process never exits from $\mathcal D_2=\{\mm_2V_2-V_2<0\}$, which is then the only region where the function $V_1$ needs to be (almost everywhere) twice differentiable in order to apply It\^o's formula.

We conclude this section with some considerations on how the theorem above will typically be used. First, when solving the system of QVIs, one deals with functions which are defined only piecewise, as it will be clear in the next section. Then, the regularity assumptions in the verification theorem will give us suitable {\color{black} smooth-pasting} conditions, leading to a system of algebraic equations. If the regularity conditions are too strong, the system has more equations than parameters, making the application of the theorem more difficult. Hence, a crucial point when stating a verification theorem is to set regularity conditions giving a solvable system of equations. In Section \ref{sec:example} we show that, in an example of one-dimensional impulse game, the regularity conditions actually lead to a well-posed algebraic system. 

\section{Examples of solvable one-dimensional impulse games}
\label{sec:example}

In Sections \ref{sec:thepb}-\ref{ssec:asymp} we provide an application of the Verification Theorem \ref{thm:verification} to an impulse game with a one-dimensional state variable modelled by a (scaled) Brownian motion, that can be shifted due to the interventions of two players with linear payoffs. We find a family of Nash equilibria for such a game and provide explicit expressions for the payoffs functions and for the optimal strategies at equilibrium. In Section \ref{ssec:moreexamples} we adapt the solving procedure to two further families of examples, with cubic payoffs and with linear and cubic payoffs, where a solution is found numerically.

\subsection{Formulation of the problem}
\label{sec:thepb}

We consider a one-dimensional real process $X$ and two players with opposite goals: player 1 prefers a high value for the process $X$, whereas the goal of player 2 is to force $X$ to take a low value. More precisely, if $x$ denotes the current value of the process, we assume that the running payoffs of the two players are given by
\begin{equation}
\label{ExPayoff}
f_1(x) = x - s_1, \qquad f_2(x)=s_2-x, \qquad s_1<s_2,
\end{equation}
where $s_1,s_2$ are fixed (possibly negative) constants.  

We assume that each player can intervene and shift $X$ from state $x$ to state $x+\delta$, with $\delta \in \rr$ \textcolor{black}{possibly varying in each intervention}. Moreover, when none of the players intervenes, we assume that $X$ follows a (scaled) Brownian motion. Hence, if $x$ denotes the initial state and $u_i = \{ (\tau_{i,k}, \delta_{i,k}) \}_{k \geq 1}$ collects the intervention times and the corresponding impulses of player $i \in \{1,2\}$, we have
\begin{equation*}
X_s = X_s^{x;u_1,u_2}= x + \sigma W_s + \sum_{k  \, : \, \tau_{1,k} \leq s} \delta_{1,k} + \sum_{k \, : \, \tau_{2,k} \leq s} \delta_{2,k},\qquad s \geq 0,
\end{equation*}
where $W$ is a standard one-dimensional Brownian motion and  $\sigma > 0$ is a fixed parameter. 

As player 2 aims at lowering the level, we can assume that her impulses are negative: $\delta_{2,k}\leq0$, for every $k \in \nn$. Similarly, we assume $\delta_{1,k}\geq0$, for every $k \in \nn$. Affecting the process has a cost for the intervening player and we also assume that there is a corresponding gain for the opponent. In our model both intervention penalties and gains consist in a fixed cost and in a variable cost, assumed to be proportional to the absolute value of the impulse: if $\phi_i$ denotes the intervention penalty for player $i$ and $\psi_j$ denotes the corresponding gain for player $j$, we assume
\begin{equation*}
\phi_i(\delta) = - c - \lambda |\delta|, 
\qquad\qquad
\psi_j(\delta) = \tilde c + \tilde \lambda |\delta|, 
\end{equation*}
where $\delta \in \rr$ is the impulse corresponding to the intervention of player $i$ and $c,\tilde c,\lambda, \tilde \lambda$ are fixed constants such that
\begin{equation*}
c \geq \tilde c \geq 0, 
\qquad\quad
\lambda \geq \tilde \lambda \geq 0,
\qquad\quad
(c,\lambda) \neq (\tilde c, \tilde \lambda).
\end{equation*}
The order conditions have this justification: if we had $c < \tilde c$ or $\lambda < \tilde \lambda$, then, for a suitable impulse $\delta$, the two players could realize a mutual gain by an (almost) instantaneous double intervention; by iterating this infinitely often in a finite interval, the two payoff functions would diverge (this phenomenon is analogous to the one already present in \cite{DeAnFerMor18} for stopping games). The condition $(c,\lambda) \neq (\tilde c, \tilde \lambda)$ will be explained in Remark \ref{RemCond} and Section \ref{ssec:asymp}. Finally, we assume
\begin{equation}
\label{ExCond}
1-\lambda \rho >0,
\end{equation}
where $\rho$ is the discount rate, the same one for both players.

This problem clearly belongs to the class described in Section \ref{sec:stochImpGame}, with
\begin{equation*}
d=1, \qquad S = \rr, \qquad \Gamma^i(x,\delta) = x + \delta, \qquad \rho_i = \rho, \qquad Z_1 =[0,\infty[, \qquad Z_2 =]-\infty,0],
\end{equation*}
and with $f_i, \phi_i, \psi_i$ as above. In short, if $\vphi_i = (\mathcal{C}_i, \xi_i)$ denotes the strategy of player $i$, the objective functions are
\begin{gather*}
J^1(x;\vphi_1,\vphi_2) \! := \! 
\eee \bigg[ \! \int_0^{\infty} \!\! e^{-\rho s} (X_s \!-\! s_1) ds  
\!-\! \sum_{k \geq 1} e^{-\rho \tau_{1,k}} (c \!+\! \lambda |\delta_{1,k}|)
\!+\! \sum_{k \geq 1} e^{-\rho \tau_{2,k}} (\tilde c \!+\! \tilde \lambda |\delta_{2,k}|)\bigg]\!,
\\
J^2(x;\vphi_1,\vphi_2) \!:=\! 
\eee \bigg[ \! \int_0^{\infty} \!\! e^{-\rho s} (s_2 \!-\! X_s) ds 
\!-\! \sum_{k \geq 1} e^{-\rho \tau_{2,k}} (c \!+\! \lambda |\delta_{2,k}|)
\!+\! \sum_{k \geq 1} e^{-\rho \tau_{1,k}} (\tilde c \!+\! \tilde \lambda |\delta_{1,k}|)\bigg]\!,
\end{gather*}
where $\{(\tau_{i,k},\delta_{i,k})\}_{k \geq 1}$ denotes the impulse control of player $i$ associated to the strategies $\vphi_1,\vphi_2$. 

As already outlined, the players have different goals: we are going to investigate if a Nash equilibrium for such a problem exists. Indeed, since $s_1<s_2$ both players gain in the interval $[s_1,s_2]$, it seems that there is room for a Nash configuration. If a Nash equilibrium exists, we denote by $V_1(x), V_2(x)$ the corresponding equilibrium payoffs with initial state $x \in \rr.$

\textcolor{black}{As a possible interpretation of the game just described, let $X$ denote the exchange rate between two currencies. The central banks of the corresponding countries (the players) have different targets for the rate: player 1 prefers a high value for $X$, while the goal of player 2 is yielding a low value. To have a tractable model, we assume that the payoffs of the two players are given, respectively, by $X-s_1$ and $s_2-X$, where $s_1, s_2$ are fixed constants with $s_2>s_1$, which leads to the one-dimensional game defined in this section. This interpretation 
corresponds to a two-player version of the model introduced and studied in, e.g., \cite{Bertola16} and \cite{Cadenillas00}.}

%
%

\subsection{Looking for candidates for the payoff functions at equilibrium}
\label{sec:candidate}

Our goal is to use the Verification Theorem \ref{thm:verification}. We start by looking for a solution to the problem in \eqref{pbNEW}, in order to get a couple of candidates $\tilde V_1, \tilde V_2$ for the payoff functions $V_1,V_2$. 

First, consider the two equations in the QVI problem \eqref{pbNEW}, that is
\begin{align*}
& \hh_i \tilde V_i-\tilde V_i=0, && \text{in} \,\,\, \{\mm_j\tilde V_j - \tilde V_j = 0\}, \\
& \max\big\{\aaa \tilde V_i -\rho \tilde V_i + f_i, \mm_i \tilde V_i- \tilde V_i \}=0, && \text{in} \,\,\, \{\mm_j \tilde V_j - \tilde V_j < 0\},
\end{align*}
for $i,j \in \{1,2\}$, with $i \neq j$; this suggests the following representation for $\tilde V_i$:
\begin{equation}\label{cand-value}
\tilde V_i(x) = 
\begin{cases}
\mm_i \tilde V_i(x), & \text{in $\{\mm_i \tilde V_i - \tilde V_i = 0 \}$,}\\
\vphi_i(x), & \text{in $\{\mm_i \tilde V_i - \tilde V_i < 0, \mm_j\tilde V_j - \tilde V_j < 0 \}$,}\\
\hh_i \tilde V_i(x), & \text{in $\{\mm_j \tilde V_j - \tilde V_j = 0 \}$,}
\end{cases}
\end{equation}
for $i \in \{1,2\}$ and $x \in \rr$, where $\vphi_i$ is a solution to
\begin{equation}
\label{EqEx}
\aaa \vphi_i - \rho \vphi_i + f_i =  \frac{\sigma^2}{2} \vphi_i''- \rho \vphi_i + f_i =0.
\end{equation}
Notice that an explicit formula for $\vphi_i$ is available: for each $x \in \rr$, we have
\begin{equation}
\label{defxi}
\begin{gathered}
\vphi_1(x) = \vphi_1^{A_{11}, A_{12}} (x) = A_{11} e^{\q x} +  A_{12} e^{- \q x} +(x-s_1)/\rho, \\
\vphi_2 (x) = \vphi_2^{A_{21}, A_{22}} (x) = A_{21} e^{\q x} +  A_{22} e^{- \q x} + (s_2-x)/\rho, 
\end{gathered}	
\end{equation}
where $A_{ij}$ are real parameters and the parameter $\theta$ is defined by
\begin{equation*}
\theta = \sqrt{\frac{2\rho}{\sigma^2}}.
\end{equation*}
In order to go on, we need to guess an expression for the intervention regions. As the goal of player 1 is to keep a high value for the process, it is reasonable to assume that her intervention region is in the form $]-\infty, \1]$, for some threshold $\1$. For a similar reason, we expect the intervention region of player 2 to be in the form $[\2,+\infty[$, for some other threshold $\2$. Since $s_1 < s_2$, we guess that $\1 < \2$; as a consequence, the real line is heuristically partitioned into three intervals:
\begin{gather*}
\text{$]-\infty, \1] = \{\mm_1 \tilde V_1 - \tilde V_1 = 0 \}$, where player 1 intervenes,} \\
\text{$]\1,\2[ = \{\mm_1 \tilde V_1 - \tilde V_1 < 0 \} \cap \{\mm_2 \tilde V_2 - \tilde V_2 < 0 \}$, where no one intervenes,} \\
\text{$[\2,+\infty[ = \{\mm_2 \tilde V_2 - \tilde V_2 = 0 \}$, where player 2 intervenes.}
\end{gather*}
By the representation (\ref{cand-value}), this leads to the following expressions for $\tilde V_1$ and $\tilde V_2$:
\begin{equation*}
\tilde V_1(x) = 
\begin{cases}
\mm_1 \tilde V_1(x), & \text{if $x \in \,\, ]-\infty, \1]$,}\\
\vphi_1(x), & \text{if $x \in \,\, ]\1,\2[$,}\\
\hh_1 \tilde V_1 (x), & \text{if $x \in [\2,+\infty[$,}
\end{cases}
\qquad
\tilde V_2(x) = 
\begin{cases}
\hh_2 \tilde V_2(x), & \text{if $x \in \,\, ]-\infty, \1]$,}\\
\vphi_2(x), & \text{if $x \in \,\, ]\1,\2[$,}\\
\mm_2 \tilde V_2 (x), & \text{if $x \in [\2,+\infty[$.}
\end{cases}
\end{equation*}
Let us now investigate the form of $\mm_i \tilde V_i$ and $\hh_i \tilde V_i$. Recall that the impulses of player 1 (resp.~player 2) are positive (resp.~negative); then, we have
\begin{gather*}
\mm_1 \tilde V_1 (x) = \sup_{\delta \geq 0} \{ \tilde V_1(x+\delta) - c - \lambda \delta \} = \sup_{y\geq x} \{ \tilde V_1(y) - c - \lambda (y-x) \}, \\
\mm_2 \tilde V_2 (x) = \sup_{\delta \leq 0} \{ \tilde V_2(x+\delta) - c - \lambda (-\delta) \} = \sup_{y\leq x} \{ \tilde V_2(y) - c - \lambda (x-y) \}.
\end{gather*}
It is reasonable to assume that the maximum point of the function $y \mapsto \tilde V_1(y) - \lambda y$ (resp.~$y \mapsto \tilde V_2(y) + \lambda y$) exists, is unique and belongs to the common continuation region $]\xx_1,\xx_2[$, where we have $\tilde V_1  =\vphi_1$ (resp.~$\tilde V_2  =\vphi_2$). As a consequence, if we denote by $x^*_i$, $i \in \{1,2\}$, such maximum points, that is
\begin{align*}
&\vphi_1(\3) = \max_{y \in ]\xx_1,\xx_2[} \{ \vphi_1(y) - \lambda y \}, \qquad \text{i.e.} \qquad \vphi_1'(\3)=\lambda, \,\,\, \vphi_1''(\3)\leq 0, \,\,\, \xx_1 \!<\! x^*_1 \!<\! \xx_2,
\\
&\vphi_2(\4) = \max_{y \in ]\xx_1,\xx_2[} \{ \vphi_2(y)  + \lambda y \}, \qquad \text{i.e.} \qquad \vphi_2'(\4)=-\lambda, \,\,\, \vphi_2''(\4)\leq 0, \,\,\, \xx_1 \!<\! x^*_2 \!<\! \xx_2,
\end{align*}
the functions $\mm_i \tilde V_i$, $\hh_i \tilde V_i$ have the following (heuristic, at the moment) expression: 
\begin{align*}
\mm_1 \tilde V_1 (x) &= \vphi_1(\3) - c - \lambda (\3-x),\quad &
\mm_2 \tilde V_2 (x) &= \vphi_2(\4) - c - \lambda (x-\4),
\\
\hh_1 \tilde V_1 (x) &= \vphi_1(x^*_2) + \tilde c + \tilde \lambda(x-x^*_2),\quad &
\hh_2 \tilde V_2 (x) &= \vphi_2(x^*_1) + \tilde c + \tilde \lambda(x^*_1-x). 
\end{align*}
As for the parameters involved in $\tilde V_1, \tilde V_2$, they must be chosen so as to satisfy the regularity assumptions in the verification theorem, which here write 
\begin{gather*}
\tilde V_1 \in C^2\big(\,]-\infty, \bar x_1[ \,\, \cup \,\, ]\bar x_1, \bar x_2[\,\big) \cap C^1\big(\,]-\infty, \bar x_2[\,\big) \cap C\big(\rr\big),
\\
\tilde V_2 \in C^2\big(\, ]\bar x_1, \bar x_2[ \,\, \cup \,\, ]\bar x_2, +\infty[ \, \big) \cap C^1\big(\, ]\bar x_1, +\infty[ \, \big) \cap C\big(\rr\big). 
\end{gather*}
Since $\tilde V_1$ and $\tilde V_2$ are, by definition, smooth in $]-\infty, \bar x_1[ \,\, \cup \,\, ]\bar x_1, \bar x_2[ \,\, \cup \,\, ]\bar x_2,+\infty[$, we have to set the parameters so that $\tilde V_i$ is continuous at $\1,\2$ and differentiable at $\bar x_i$ (we underline that $\tilde V_1$ and $\tilde V_2$ might not be differentiable at, respectively, $\bar x_2$ and $\bar x_1$). 

Finally, to summarize all the previous arguments, our candidates for the payoff functions of some Nash equilibrium are defined as follows.
\begin{defn}
	\label{def:candidate}
	For every $x \in \rr$, we set
	\begin{equation}
	\label{Vpratic}
	\begin{gathered}
	\tilde V_1(x) = 
	\begin{cases}
	\vphi_1(\3) - c - \lambda (\3-x), & \text{if $x \in \,\, ]-\infty, \1]$,}\\
	\vphi_1(x), & \text{if $x \in \,\,]\1,\2[$,}\\
	\vphi_1(\4) + \tilde c + \tilde \lambda(x-\4), & \text{if $x \in [\2,+\infty[$,}
	\end{cases}
	\\
	\tilde V_2(x) = 
	\begin{cases}
	\vphi_2(\3) + \tilde c + \tilde \lambda(\3-x), & \text{if $x \in \,\,]-\infty, \1]$,}\\
	\vphi_2(x), & \text{if $x \in \,\,]\1,\2[$,}\\
	\vphi_2(\4) - c - \lambda (x-\4), & \text{if $x \in [\2,+\infty[$,}
	\end{cases}
	\end{gathered}		
	\end{equation}
	where $\vphi_1 = \vphi_1^{A_{11}, A_{12}}$, $\vphi_2 = \vphi_2^{A_{21}, A_{22}}$ and the eight parameters involved
	\begin{equation*}
	(A_{11}, A_{12}, A_{21}, A_{22}, \1, \2, \3, \4)
	\end{equation*}
	satisfy the order conditions
	\begin{equation}
	\label{OrdCondTWOPL}
	\1 < \3 < \2, \qquad \1 < \4 < \2, 
	\end{equation}
	and the following conditions:
	\begin{subnumcases}{\label{syst1}}
	\vphi_1'(\3)=\lambda \quad \text{and} \quad \vphi_1''(\3)\leq0 ,  & \textit{(optimality of $\3$)} \label{syst1-a} \\ 
	\vphi_1'(\1)=\lambda,  & \textit{($C^1$-pasting in $\1$)} \label{syst1-b}\\ 
	\vphi_1(\1)= \vphi_1(\3) - c - \lambda (\3-\1),  & \textit{($C^0$-pasting in $\1$)} \label{syst1-c} \\ 
	\vphi_1(\2)= \vphi_1(\4) + \tilde c + \tilde \lambda(\2-\4),  & \textit{($C^0$-pasting in $\2$)} \label{syst1-d}
	\end{subnumcases} 
	\begin{subnumcases}{\label{syst2}}
	\vphi_2'(\4)=-\lambda \quad \text{and} \quad \vphi_2''(\4)\leq0 ,  & \textit{(optimality of $\4$)} \label{syst2-a} \\ 
	\vphi_2'(\2)=-\lambda,  & \textit{($C^1$-pasting in $\2$)} \label{syst2-b} \\ 
	\vphi_2(\1)= \vphi_2(\3) + \tilde c + \tilde \lambda (\3-\1),  & \textit{($C^0$-pasting in $\1$)} \label{syst2-c} \\ 
	\vphi_2(\2)= \vphi_2(\4) - c - \lambda(\2-\4).  & \textit{($C^0$-pasting in $\2$)} \label{syst2-d}
	\end{subnumcases}
\end{defn}

In order to have a well-posed definition, we need to show that the conditions in \eqref{OrdCondTWOPL}-\eqref{syst1}-\eqref{syst2} actually admit a solution. Indeed, we can here prove that there exists a family of solutions to \eqref{OrdCondTWOPL}-\eqref{syst1}-\eqref{syst2}, i.e., that we have infinitely many (candidates for the) Nash equilibria and corresponding payoffs.

\begin{prop}
	\label{PropEx}
	There exists a family of 8-uples $(A_{11}, A_{12}, A_{21}, A_{22}, \1, \2, \3, \4)$ satisfying the conditions in \eqref{OrdCondTWOPL}-\eqref{syst1}-\eqref{syst2}. Moreover, for each of such 8-uples, there exists $\tilde x \in ]x^*_2, \bar x_2[$ such that $\vphi_2''<0$ in $]\bar x_1, \tilde x[$ and  $\vphi_2''>0$ in $]\tilde x, \bar x_2[$.
\end{prop}

\begin{proof}
	First, we reduce the number of equations. Notice that, for any $\tilde s \in \rr$ and $x \in \rr$, the running costs $(f_1,f_2)$ satisfy
	\begin{equation*}
	f_1(x)=f_2(2\tilde s - x)+2\tilde s -(s_1 + s_2).
	\end{equation*}
	We guess a corresponding relation for $(\vphi_1, \vphi_2)$, that is
	\begin{equation*}
	\vphi_1(x)=\vphi_2(2\tilde s - x)+\frac{2\tilde s -(s_1 + s_2)}{\rho},
	\end{equation*}
	and we look for couples $(\bar x_1, \bar x_2)$, $(x^*_1, x^*_2)$ symmetric with respect to $\tilde s$. Hence, we focus on candidates such that
	\begin{equation}
	\label{symmetry}
	\bar x_1 = 2 \tilde s - \bar x_2, \quad \,\,\, 
	x^*_1 = 2 \tilde s - x^*_2, \quad \,\,\,
	A_{11} =A_{22}e^{-2 \theta \tilde s}, \quad \,\,\,
	A_{12} =A_{21}e^{2\theta \tilde s}.
	\end{equation}
	Under condition \eqref{symmetry}, the systems in \eqref{syst1} and \eqref{syst2} are independent and equivalent: namely, the 4-uple $(A_{11}, A_{12}, \bar x_1, x^*_1)$ solves \eqref{syst1} if and only if $(A_{21}, A_{22}, \bar x_2, x^*_2)$, defined by \eqref{symmetry}, is a solution to \eqref{syst2}. Hence, we just need to solve one of the two systems of equations: we decide to focus on \eqref{syst2}, along with the order condition \eqref{OrdCondTWOPL}. By the change of variable
	\begin{equation}
	\label{Eq0}
	\bar y = e^{\theta \left(\bar x_2 - \tilde s \right)}, \quad \,\,\,
	y^* = e^{\theta \left(x^*_2 - \tilde s \right)}, \quad \,\,\,
	A_1= 2\theta A_{21} e^{\theta \tilde s}, \quad \,\,\,
	A_2= 2\theta A_{22} e^{-\theta \tilde s}
	\end{equation}
	and some algebraic manipulations, the conditions in \eqref{OrdCondTWOPL} and \eqref{syst2} read ({\color{black}see Appendix \ref{SecApp} for the details}, we set $\eta = (1-\lambda \rho)/\rho$, notice that $\eta >0$)
	\begin{subnumcases}{\label{news}}
	A_1 (y^*)^2 - 2 \eta y^* - A_2 = 0,  \label{news-a} \\ 
	A_1 \bar y^2 - 2 \eta \bar y - A_2 = 0, \label{news-b} \\ 
	(A_1+A_2)^2 (\bar y - y^*) + 2 A_2 \big[\theta(c-\tilde c) + (\lambda - \tilde \lambda) \log(\bar y / y^*)\big]=0, \label{news-c}\\ 
	A_1(\bar y - y^*) + \theta c - \eta \log(\bar y / y^*)=0, \label{news-d}\\
	y^*>0, \quad \bar y>0, \quad y^*<\bar y, \quad 1<\bar y y^*, \quad A_1y^* - \eta \leq 0. \label{news-e}
	\end{subnumcases}
	
	We now prove that there exists a unique solution $(A_1,A_2,\bar y,y^*)$ to \eqref{news}. Given a fixed pair $(A_1,A_2) \in \aaa$, where
	\begin{equation}
	\aaa = \big\{ (A_1, A_2) \, : \, A_1 >0, \quad A_2<0, \quad A_1+A_2<0, \quad \eta^2 + A_1 A_2>0 \big\},
	\end{equation}
	there exists a unique solution to \eqref{news-a}-\eqref{news-b}-\eqref{news-e}, given by
	\begin{equation}
	\label{proof2}
	\bar y (A_1,A_2) = \frac{\eta + \sqrt{\eta^2+A_1 A_2}}{A_1},
	\qquad\qquad
	y^*(A_1,A_2) = \frac{\eta - \sqrt{\eta^2+A_1 A_2}}{A_1}.
	\end{equation}	
	To conclude, we just need to prove that there exists a unique pair $(A_1,A_2)$ such that 
	\begin{equation}
	\label{proof3}
	\text{$(A_1,A_2) \in \aaa$ and $\big(A_1,A_2, \bar y (A_1,A_2), y^*(A_1,A_2) \big)$ is a solution to \eqref{news-c}-\eqref{news-d},}
	\end{equation}
	that is, by the expressions in \eqref{proof2}, such that
	\begin{subnumcases}{\label{proof5}}
	(A_1+A_2)^2 \sqrt{\eta^2+ A_1 A_2} + A_1 A_2 \left[\theta(c-\tilde c) + (\lambda - \tilde \lambda) \log\left( \frac{\eta + \sqrt{\eta^2+ A_1 A_2}}{\eta - \sqrt{\eta^2+ A_1 A_2}} \right)\right] =0, \qquad\qquad \label{proof5-a}\\
	2 \sqrt{\eta^2+ A_1 A_2} + \theta c - \eta \log\left( \frac{\eta + \sqrt{\eta^2+ A_1 A_2}}{\eta - \sqrt{\eta^2+ A_1 A_2}} \right)=0, \label{proof5-b}\\
	A_1 >0, \quad A_2<0, \quad A_1+A_2<0, \quad \eta^2 + A_1 A_2>0. \label{proof5-c}
	\end{subnumcases}
	{\color{black} Namely, \eqref{proof5-a}-\eqref{proof5-b} correspond to \eqref{news-c}-\eqref{news-d}, whereas the inequalities in \eqref{proof5-c} correspond to the condition $(A_1,A_2) \in \aaa$.}
	
	For $x \in (0, \eta)$, define the function
	\begin{equation}
	\label{DefF}
	F(x)= 2 x + \theta c - \eta \log\left( \frac{\eta + x}{\eta - x} \right).
	\end{equation}
	Since $F(0^+)= \theta c>0$, $F(\eta^-)=-\infty$ and $F'<0$, there exists a unique $\xi \in (0, \eta)$ such that $F(\xi)=0$. Consequently, \eqref{proof5} is equivalent to {\color{black} (see Appendix \ref{SecApp} for the details)}
	\begin{equation}
	\label{proof6}
	\begin{cases}
	A_1A_2 = -M, \\
	A_1+A_2=-2N, \\
	A_1 >0, \quad A_2<0,
	\end{cases}
	\text{with} \,\,
	M=\eta^2-\xi^2
	\,\,\text{and}\,\,
	N=\sqrt{\frac{(\eta^2 - \xi^2)\big[\theta\eta(c-\tilde c) + (\lambda -  \tilde\lambda) (2\xi + \theta c)\big]}{4\eta\xi}},
	\end{equation}
	which trivially has a unique solution (notice that $N^2+M>0$), namely
	\begin{equation}
	\label{proof7}
	A_1 = -N + \sqrt{N^2+M},
	\qquad\qquad
	A_2 = -N - \sqrt{N^2+M}.
	\end{equation}
	
	Finally, it is immediate to see that $\vphi_2''<0$ in $]-\infty,\tilde x[$ and $\vphi_2''>0$ in $]\tilde x, + \infty[$, for a suitable $\tilde x \in \rr$. By the change of variable, $\vphi_2''(\bar x_2)>0$ (resp.~$\vphi_2''(x^*_2)<0$) if and only if $A_1 \bar y - \eta >0$ (resp.~$A_1 y^* - \eta <0$), which is trivially true; hence, $\tilde x \in ]x^*_2,\bar x_2[$.
\end{proof}

\begin{remark}
\emph{From the proof of Proposition \ref{PropEx}, we see that the system in \eqref{syst2} has more than one solution, but only one satisfies the order condition \eqref{OrdCondTWOPL}. In particular, we notice that the other solution of \eqref{syst2} corresponds to
$\tilde A_{1} = - A_2$, $\tilde A_{2} = - A_1$, $\tilde y^*=1/\bar y$,  $\tilde{\bar y}= 1/y^*$.}
\end{remark}

\begin{remark}
\label{RemForm}
\emph{There are infinitely many solution to the system \eqref{OrdCondTWOPL}-\eqref{syst1}-\eqref{syst2}, indexed by the parameter $\tilde s \in \rr$. To simplify the notation, we will often omit the dependence on such a parameter and write, for example, $\bar x_i$ instead of $\bar x_i(\tilde s)$. By combining \eqref{symmetry}, \eqref{Eq0}, \eqref{proof2} and \eqref{proof7}, we can get (semi-)explicit formulas for the 8-uples $(A_{11}, A_{12}, A_{21}, A_{22}, \1, \2, \3, \4)$ which solve \eqref{OrdCondTWOPL}-\eqref{syst1}-\eqref{syst2}: namely, for any $\tilde s \in \rr$ we have
	\begin{equation}
	\label{explicit}
	\begin{gathered}
	\bar x_i = \tilde s + \frac{(-1)^i}{\theta} \log \left[\sqrt{\frac{\eta + \xi}{\eta - \xi}} \left( \sqrt{\Gamma+1} + \sqrt{\Gamma} \right) \right],
	\\
	x^*_i = \tilde s + \frac{(-1)^i}{\theta} \log \left[ \sqrt{\frac{\eta - \xi}{\eta + \xi}} \left( \sqrt{\Gamma +1} + \sqrt{\Gamma}\right) \right],
	\\
	A_{ij} = e^{(-1)^j\theta \tilde s} \frac{\sqrt{\eta^2-\xi^2}}{2\theta} \bigg( (-1)^{i+j+1}\sqrt{\Gamma+1} - \sqrt{\Gamma}\bigg),
	\end{gathered}
	\end{equation}
	for $i,j \in \{1,2\}$, where $\xi=\xi(c,\theta,\eta) \in (0,\eta)$ is the unique zero of the function $F$ in \eqref{DefF} and the coefficients are defined by 
	\begin{equation}
	\label{coeffexplicit}
	\theta = \sqrt{\frac{2\rho}{\sigma^2}},
	\qquad
	\eta = \frac{1 -\lambda\rho}{\rho},
	\qquad
	\Gamma = \frac{\theta(c-\tilde c)}{4\xi} + \frac{\theta c(\lambda -  \tilde\lambda)}{4\eta\xi} + \frac{\lambda -  \tilde\lambda}{2\eta}.
	\end{equation}
Also, notice that \eqref{symmetry} implies that 
\begin{equation}
\label{RelationV}
V_1(x)=V_2(2\tilde s - x)+\frac{2\tilde s -(s_1 + s_2)}{\rho},
\end{equation}
for $x \in \rr$. In particular, when $\tilde s = (s_1+s_2)/2$ the functions $\tilde V_i$ are symmetric with respect to $\tilde s$.}
\end{remark}

\begin{remark}
\label{RemCond}
\emph{Let $i,j \!\in\! \{1,2\}$, with $i \!\neq\! j$. From \eqref{explicit}, we notice that $\bar x_i = x_j^*$ when $\Gamma = 0$, which happens if and only if $(c,\lambda) = (\tilde c,\tilde \lambda)$. This situation gives rise to a degenerate solution, where players intervene infinitely often in each instant. We analyze this in more detail in Section 4.4, where we study the case where $\lambda = \tilde \lambda$ and $c \to \tilde c^+$.} 
\end{remark}

%
%

\subsection{Application of the verification theorem}
\label{sec:applying}

We now apply the Verification Theorem \ref{thm:verification} and prove that the candidates $\tilde V_1, \tilde V_2$ in Definition \ref{def:candidate} actually coincide with the payoff functions $V_1, V_2$ of the problem described in Section \ref{sec:thepb}. We refer the reader to Section \ref{ssec:QVI} for the definition of the functions $\delta_1, \delta_2, \mm_1,\mm_2$ used in the following lemma. 

\begin{lemma}
	\label{lem:studyofMV}
	Let $\tilde V_1, \tilde V_2$ be as in Definition \ref{def:candidate}. For every $x \in \rr$ we have
	\begin{equation}
	\label{formuladelta}
	\delta_1(x) = 
	\begin{cases}
	x^*_1 - x, & \text{in $]-\infty, x^*_1]$,}\\
	0, & \text{in $]x^*_1, +\infty[$,}
	\end{cases}
	\qquad\quad
	\delta_2(x) = 
	\begin{cases}
	0, & \text{in $]-\infty, x^*_2[$,}\\
	x^*_2-x, & \text{in $[x^*_2, +\infty[$.}
	\end{cases}
	\end{equation}
	Moreover, we have		
	\begin{equation}
	\label{contregion}
	\begin{aligned}
	\mm_1 \tilde V_1 - \tilde V_1 &\leq 0, \,\,\, &
	\{\mm_1 \tilde V_1 - \tilde V_1 < 0 \} &= \,\, ]\bar x_1, +\infty[, \,\,\, &
	\{\mm_1 \tilde V_1 - \tilde V_1 = 0 \} &= \,\, ]\!-\! \infty, \bar x_1], 
	\\	
	\mm_2 \tilde V_2 - \tilde V_2 &\leq 0, \,\,\, &
	\{\mm_2 \tilde V_2 - \tilde V_2 < 0 \} &= \,\, ]\!-\! \infty,\bar x_2[, \,\,\, &
	\{\mm_2 \tilde V_2 - \tilde V_2 = 0 \} &= [\bar x_2, + \infty[.
	\end{aligned}
	\end{equation} 
\end{lemma}

\begin{proof}
	We give the proof only for $\delta_2$ and $\mm_2 \tilde V_2$, the arguments for $\delta_1$ and $\mm_1 \tilde V_1$ being the same. For every $x \in \rr$, we have
	\begin{equation}\label{comp-M2}
	\mm_2 \tilde V_2(x) \!=\! \max_{\delta_2 \leq 0} \{ \tilde V_2(x+\delta_2) - c - \lambda(-\delta_2) \} \!=\! \max_{y \leq x} \{ \tilde V_2(y) - c - \lambda(x-y) \} \!=\!  \max_{y \leq x} \{\Gamma_2(y)\} - c -\lambda x,
	\end{equation}
	where for each $y \in \rr$ we have set
	\begin{equation*}
	\Gamma_2(y) = \tilde V_2(y) + \lambda y.
	\end{equation*}
	By the definition of $\tilde V_2$, we have $\Gamma_2'(x_2^*)=\Gamma_2'(\bar x_2)=0$. Moreover, we notice that:
	\begin{itemize}
		\item[-] $\Gamma_2'= \lambda - \tilde \lambda \geq 0$ in $]-\infty,\bar x_1[$, by the definition of $\tilde V_2$;
		\item[-] $\Gamma_2'> 0$ in $]\bar x_1, x_2^*[$, as $\Gamma_2'(x_2^*)=0$  and $\Gamma_2'$ is decreasing in $]\bar x_1, x_2^*[$ (since, by Proposition \ref{PropEx}, we have $\Gamma_2''= \vphi_2''<0$ in $]\bar x_1, x_2^*[$);
		\item[-] $\Gamma_2'< 0$ in $]x_2^*, \bar x_2[$, as $\Gamma_2'(x_2^*)=\Gamma_2'(\bar x_2)=0$ and, in the interval $]x_2^*, \bar x_2[$, $\Gamma_2'$ is first decreasing and then increasing (since, by Proposition \ref{PropEx}, $\Gamma_2''= \vphi_2''$ is negative in $]x^*_2, \tilde x[$ and positive in $]\tilde x, \bar x_2[$);
		\item[-] $\Gamma_2' = 0$ in $]\bar x_2, + \infty[$, by the definition of $\tilde V_2$.
	\end{itemize}
	As a consequence, the function $\Gamma_2$ has a unique global maximum point in $x^*_2$, so that 
	\begin{equation*}
	\max_{y\leq x} \Gamma_2(y) = 
	\begin{cases}
	\Gamma_2(x), & \text{in $]-\infty, x^*_2]$,}\\
	\Gamma_2(x_2^*), & \text{in $]x^*_2, +\infty[$;}
	\end{cases}
	\end{equation*}
	therefore, by the computations in (\ref{comp-M2}), we have
	\begin{equation*}
	\mm_2 \tilde V_2(x) = 
	\begin{cases}
	\tilde V_2(x) - c, & \text{in $]-\infty, x^*_2]$,}\\
	\vphi_2(x^*_2) - c- \lambda (x-x^*_2), & \text{in $]x^*_2, +\infty[$,}
	\end{cases}
	\end{equation*}
	as $\tilde V_2(x^*_2) \!=\! \vphi_2(x^*_2)$, since $x^*_2 \in ]\bar x_1, \bar x_2[$. By the definition of $\tilde V_2$, this can be written as
	\begin{equation*}
	\mm_2 \tilde V_2(x) = 
	\begin{cases}
	\tilde V_2(x) - \xi_2(x), & \text{in $]-\infty, \bar x_2[$,}\\
	\tilde V_2(x) , & \text{in $[\bar x_2,+\infty[$,}
	\end{cases}
	\end{equation*}
	where, for each $x \in ]-\infty, \bar x_2[$, we have set
	\begin{equation*}
	\xi_2(x)  = 
	\begin{cases}
	c, & \text{in $]-\infty, x^*_2[$,}\\
	\vphi_2(x) - \vphi_2(x^*_2) + c +  \lambda (x-x^*_2), & \text{in $[x^*_2,\bar x_2[$.}
	\end{cases}
	\end{equation*}
	Let us prove that $\xi_2 >0$. Recall by \eqref{syst2} that $\vphi_2(\2)= \vphi_2(\4) - c - \lambda(\2-\4)$. Then, if $x \in [x^*_2,\bar x_2[$ we have that
	\begin{equation*}
	\vphi_2(x) - \vphi_2(x^*_2) + c +  \lambda (x-x^*_2) = \vphi_2(x) - \vphi_2(\bar x_2) - \lambda (\bar x_2 -x) = \Gamma_2(x) - \Gamma_2(\bar x _2) >0,
	\end{equation*} 
	as $\Gamma_2$ is decreasing in $[x^*_2,\bar x_2[$. Hence, $\xi_2$ is strictly positive, so that \eqref{contregion} holds. Finally, by the previous arguments it is clear that
	\begin{equation*}
	\argmax_{\delta_2 \leq 0} \{ \tilde V_2(x+\delta_2) - c - \lambda|\delta_2| \} = 
	\begin{cases}
	\,\{0\}, & \text{in $]-\infty, x^*_2[$,}\\
	\,\{x^*_2-x\}, & \text{in $]x^*_2, +\infty[$,}
	\end{cases}
	\end{equation*}
	which implies \eqref{formuladelta}.
\end{proof}

\begin{prop}
	\label{prop:checkverification}
	For $\tilde s \in \rr$, let  $x^*_i = x^*_i(\tilde s)$ and $\bar x_i = \bar x_i (\tilde s)$, with $i \in \{1,2\}$, as in Definition \ref{def:candidate}. Then, a Nash equilibrium for the problem in Section \ref{sec:thepb} is given by the strategies $(\mathcal{C}^*_1,\xi^*_1)$, $(\mathcal{C}^*_2,\xi^*_2)$ defined by
	\begin{align*}
	\mathcal{C}_1^* &= \,\, ]\bar x_1, \, +\infty[,&  
	\xi^*_1(y) &= x^*_1 - y , \\
	\mathcal{C}_2^* &= \,\, ]\!-\!\infty, \, \bar x_2[,&  
	\xi^*_2(y) &= x^*_2 - y,
	\end{align*}
	with $y \in \rr$. Moreover, the functions $\tilde V_1, \tilde V_2$ in Definition \ref{def:candidate} coincide with the equilibrium payoff functions $V_1,V_2$:
	\begin{equation*}
	V_1 \equiv \tilde V_1 \qquad \text{and} \qquad V_2 \equiv \tilde V_2.
	\end{equation*}
\end{prop}

\begin{remark}
	\emph{We underline that $(\mathcal{C}^*_1,\xi^*_1)$, $(\mathcal{C}^*_2,\xi^*_2)$ depends on the free parameter $\tilde s$, i.e., $\mathcal{C}^*_i = \mathcal{C}^*_i(\tilde s)$ and $\xi^*_i = \xi^*_i(\cdot \, ; \tilde s)$  (we often omit to underline the dependence to simplify the notations). In particular, there exist infinitely many Nash equilibria, indexed by the parameter $\tilde s \in \rr$. Notice that the corresponding optimal intervention regions and intervention functions consist in the translations of a fixed interval: $\mathcal{C}_i^*(\tilde s) = \tilde s + \mathcal{C}_i^*(0)$ and $\xi_i^*(\cdot \,; \tilde s) = \tilde s + \xi_i^*(\cdot \,; 0)$, for any $s \in \rr$.}
\end{remark}

\begin{remark}
\emph{Recall the practical characterization of the strategy: if $x$ is the current state of the process, player 1 (resp.~player 2) intervenes when $x \leq \bar x_1$ (resp.~$x \geq \bar x_2$) and moves the process to the new state $x^*_1$ (resp.~$x^*_2$).}
\end{remark}

\begin{proof}
	We have to check that the candidates $\tilde V_1, \tilde V_2$  satisfy all the assumptions of Theorem \ref{thm:verification}. We prove the claim for $\tilde V_2$, the arguments for $\tilde V_1$ being the same.	For the reader's convenience, we briefly report the conditions we have to check:
	\begin{itemize}
		\item[(i)] $\tilde V_2 \in C^2(]\bar x_1, +\infty[ \setminus \{\bar x_2\}) \cap C^1(]\bar x_1, +\infty[) \cap C(\rr)$ and has polynomial growth;
		\item[(ii)] $\mm_2 \tilde V_2 - \tilde V_2 \leq 0$;
		\item[(iii)] in $\{\mm_1 \tilde V_1 - \tilde V_1 = 0 \}$ we have $\tilde V_2 = \hh_2 \tilde V_2$;
		\item[(iv)] in $\{\mm_1 \tilde V_1 - \tilde V_1 < 0 \}$ we have $\max\big\{\aaa \tilde V_2 -\rho \tilde V_2+ f_2, \mm_2 \tilde V_2-\tilde V_2 \}=0$;
		\item[(v)] the equilibrium strategies are $x$-admissible (see Definition \ref{AdmStrat}) for every $x \in \rr$.
	\end{itemize}
	
	\textit{Condition (i) and (ii).} The first condition holds by the definition of $\tilde V_2 $, whereas the second condition has been proved in \eqref{contregion}.
	
	\textit{Condition (iii).} Let $x \in \{\mm_1 \tilde V_1 - \tilde V_1 = 0 \} = ]-\infty, \bar x_1] $. By the definition of $\hh_2 \tilde V_2$ in \eqref{MH}, by \eqref{formuladelta} and by the definition of $\tilde V_2$ we have 
	\begin{equation*}
	\hh_2 \tilde V_2(x) = \tilde V_2 (x+\delta_1(x)) + \tilde c + \tilde\lambda|\delta_1(x)| = \tilde V_2(x^*_1) + \tilde c + \tilde\lambda(x^*_1 - x) = \tilde V_2(x),
	\end{equation*}
	where we have used that $\tilde V_2(x^*_1) = \vphi_2(x^*_1)$, since $x^*_1 \in ]\bar x_1, \bar x_2[$.
	
	\textit{Condition (iv).} We have to prove that 
	\begin{equation*}
	\max\big\{\aaa \tilde V_2 -\rho \tilde V_2+ f_2, \mm_2 \tilde V_2- \tilde V_2 \}=0,
	\qquad \text{in} \,\, \{\mm_1 \tilde V_1 - \tilde V_1 < 0 \} = ]\bar x_1, +\infty[.
	\end{equation*}
	In $]\bar x_1, \bar x_2[$ the claim is true, as $\mm_2 \tilde V_2- \tilde V_2 < 0$ by \eqref{contregion} and $\aaa \tilde V_2 -\rho \tilde V_2+ f_2 = 0$ by definition (in $]\bar x_1, \bar x_2[$ we have $\tilde V_2 = \vphi_2$, which is a solution to the ODE \eqref{EqEx}). In $[\bar x_2, \infty[$ we already know by \eqref{contregion} that $\mm_2 \tilde V_2- \tilde V_2 = 0$. Then, to conclude we have to check that
	\begin{equation*}
	\aaa \tilde V_2(x) -\rho \tilde V_2(x)+ f_2(x) \leq 0, \qquad \text{$\forall x \in [\bar x_2, \infty[$.}
	\end{equation*}
	As $\tilde V_2(x) = \vphi_2(\4) - c - \lambda(x-\4)$ by the definition of $\tilde V_2(x)$, the inequality can be written as
	\begin{equation*}
	-\rho \big( \vphi_2(\4) - c - \lambda(x-\4) \big) +f_2(x) \leq 0, \qquad \text{$\forall x \in [\bar x_2, \infty[$.}
	\end{equation*}
	Since $\vphi_2(\2)= \vphi_2(\4) - c - \lambda(\2-\4)$ by \eqref{syst2}, we can rewrite the claim as
	\begin{equation*}
	-\rho \big( \vphi_2(\bar x_2) - \lambda (x-\bar x_2) \big) +f_2(x) \leq 0, \qquad \text{$\forall x \in [\bar x_2, \infty[$.}
	\end{equation*}
	The function $x \mapsto \lambda \rho x + f_2(x) = (\lambda \rho-1) x + s_2$ is decreasing, hence it is enough to prove the claim in $x=\bar x_2$:
	\begin{equation*}
	-\rho \vphi_2(\bar x_2) +f_2(\bar x_2) \leq 0.
	\end{equation*}
	Since $\aaa \vphi_2(\bar x_2) - \rho \vphi_2(\bar x_2) + f_2(\bar x_2) =0$, we can rewrite as
	\begin{equation*}
	-\frac{\sigma^2}{2} \vphi_2''(\bar x_2) \leq 0,
	\end{equation*}
	which is true since $\vphi_2''(\bar x_2) \geq 0$ by Proposition \ref{PropEx}.
	
	\textit{Condition (v).} Let $x$ be the initial state of the process. By construction the controlled process never exits from $]\bar x_1, \bar x_2[ \,\,\cup\, \{x\}$, so that condition \eqref{L1process} holds. It is easy to check that all the other conditions of Definition \ref{AdmStrat} are satisfied. The only non-trivial proof is the integrability of the intervention costs: let us prove that for $i \in \{1,2\}$ we have (the result for $\tilde c, \tilde \lambda$ immediately follows, as $\tilde \lambda < \lambda$ and $\tilde c < c$)
	\begin{equation}
	\label{intergrability}
	\eee \bigg[ \sum_{k \geq 1} e^{-\rho \tau^*_{i,k}} (c + \lambda |\delta^*_{i,k}|) \bigg] < \infty, 
	\end{equation}
	where $ \{ \tau^*_{i,k}, \delta^*_{i,k}\}_k$ are the controls corresponding to the equilibrium strategies. 
	
	To start, let us assume that the initial state $x$ is either $x^*_1$ or $x^*_2$.  We here consider $x=x^*_1$, the arguments are the same in the case $x=x^*_2$. Since player $i$ shifts the process to $x^*_i$ when the state $\bar x_i$ is hit, the idea is to write $\tau^*_{i,k}$ as a sum of independent exit times. First of all, we re-label the indexes and write $\{\tau^*_{i,k}\}_{i,k}$ as $\{\sigma_j\}_j$, with $\sigma_j < \sigma_{j+1}$ for every $j \in \nn$.
	Denote by $\mu_i$ the exit time of the process $x^*_i + \sigma W$ from $]\bar x_1,\bar x_2[$, where $W$ is a real Brownian motion; then, each time $\sigma_j$ can be written as $\sigma_j = \sum_{l=1}^j \zeta_l$, where the $\zeta_l$ are independent variables which are distributed either as $\mu_1$ or as $\mu_2$. We can now estimate \eqref{intergrability}. As $\delta^*_{i,k} \in \{\bar x_2 - x^*_2, x^*_1- \bar x_1\}$, we have 
	\begin{equation*}
	\mathbb{E}_{x^*_1} \bigg[ \sum_{i \in \{1,2\}}\sum_{k \geq 1} e^{-\rho \tau^*_{i,k}} (c + \lambda |\delta^*_{i,k}|) \bigg]
	\leq (c + \lambda \max \{\bar x_2 - x^*_2, x^*_1- \bar x_1\}) \, \mathbb{E}_{x^*_1} \bigg[ \sum_{i \in \{1,2\}}\sum_{k \geq 1} e^{-\rho \tau^*_{i,k}} \bigg].
	\end{equation*}
	By the definition of $\{\sigma_j\}_j$ and the decomposition of $\sigma_j$,
	\begin{equation*}
	\mathbb{E}_{x^*_1} \bigg[ \sum_{i \in \{1,2\}}\sum_{k \geq 1} e^{-\rho \tau^*_{i,k}} \bigg] \!=\! \mathbb{E}_{x^*_1} \bigg[ \sum_{j \geq 1} e^{-\rho \sigma_j} \bigg] \!=\! \mathbb{E}_{x^*_1} \bigg[ \sum_{j \geq 1} e^{-\rho \sum_{l=1}^{j}\zeta_l} \bigg] \!=\! \mathbb{E}_{x^*_1} \bigg[ \sum_{j \geq 1} \prod_{l=1,\dots,j} e^{-\rho \zeta_l} \bigg].
	\end{equation*}
	By the Fubini-Tonelli theorem and the independence of the variables $\zeta_j$, we get 
	\begin{equation*}
	\mathbb{E}_{x^*_1} \bigg[ \sum_{j \geq 1} \prod_{l=1,\dots,j} e^{-\rho \zeta_l} \bigg] = \sum_{j \geq 1} \prod_{l=1,\dots,j} \mathbb{E}_{x^*_1} [ e^{-\rho \zeta_l} ] \leq \sum_{j \geq 1} \big( \mathbb{E}_{x^*_1} [ e^{-\rho \min \{\mu_1, \mu_2\}}] \big)^j,
	\end{equation*}
	which is a converging geometric series, since $\mu_1,\mu_2 > 0$ ($\mu_i$ is strictly positive since $\bar x_1 < x^*_i< \bar x_2$). To sum up, we have shown  
	\begin{equation*}
	\mathbb{E}_{x^*_1} \bigg[ \sum_{i \in \{1,2\}}\sum_{k \geq 1} e^{-\rho \tau^*_{i,k}} (\max \{ c, \tilde c \} + \lambda |\delta^*_{i,k}|) \bigg] < \infty,
	\end{equation*}
	which clearly implies \eqref{intergrability}. The general case with initial state $x \in \rr$ can be treated similarly: we have $\sigma_j = \eta + \sum_{l=1}^j \zeta_l$, where $\eta$ is the exit time of $x + \sigma W$ from $[\bar x_1,\bar x_2]$, and the argument can be easily adapted. 
\end{proof}

\subsection{Comments and some limit properties}	
\label{ssec:asymp}

In order to understand the qualitative behaviour of the Nash equilibria described in the previous section, we here study some asymptotic properties of the corresponding continuation regions and payoff functions. First, we recall some formulas from the previous sections, for reader's convenience. The payoff functions of some of the Nash equilibria described before are
\begin{equation}
V_2(x) = 
\begin{cases}
\vphi_2^{A_{21},A_{22}}(\3) + \tilde c + \tilde \lambda(\3-x), & \text{if $x \in \,\,]-\infty, \1]$,}\\
\vphi_2^{A_{21},A_{22}}(x), & \text{if $x \in \,\,]\1,\2[$,}\\
\vphi_2^{A_{21},A_{22}}(\4) - c - \lambda (x-\4), & \text{if $x \in [\2,+\infty[$,}
\end{cases}
\qquad
V_1(x)\!=\!V_2(2\tilde s - x)+\frac{2\tilde s \!-\!(s_1 \!+\! s_2)}{\rho},
\end{equation}
where the function $\vphi_2^{A_{21}, A_{22}}$ is defined in \eqref{defxi} and the parameters $\bar x_i, x^*_i, A_{ij}$ are defined in \eqref{explicit}. In particular, we recall the symmetry relations: 
\begin{equation*}
\bar x_1 = 2 \tilde s - \bar x_2,
\qquad\qquad
x^*_1 = 2 \tilde s - x^*_2.
\end{equation*}
Also, recall that player 1 (resp.~player 2) intervenes if the state is smaller than $\bar x_1$ (resp.~greater than $\bar x_2$) and moves the process to $x^*_1$ (resp.~$x^*_2$).

Finally we remark that the parameter $\xi=\xi(c, \theta, \eta)$, defined as the unique zero of the function $F$ in \eqref{DefF}, satisfies the following properties: for given parameters $\theta$ and $\eta$, the function $c \mapsto \xi(c, \theta, \eta)=:\xi(c)$ belongs to $C^\infty(]0,\infty[)$ and we have
\begin{equation}
\label{derivXi}
\xi'(c) = \frac{\theta}{2} \frac{\eta^2 -\xi^2(c)}{\xi^2(c)},
\qquad\qquad
\xi''(c) = - \theta \eta^2 \frac{\xi'(c)}{\xi^3(c)}= - \frac{\theta^2 \eta^2}{2} \frac{\eta^2 -\xi^2(c)}{\xi^5(c)};
\end{equation}
in particular, the following limits hold:
\begin{equation}
\label{limxi}
\lim_{c \to 0^+} \xi(c)=
\lim_{c \to 0^+} \frac{c}{\xi(c)} =
\lim_{c \to 0^+} c \, \xi'(c) = 
\lim_{c \to +\infty} c(\eta - \xi(c)) = 0,
\qquad\qquad
\lim_{c \to +\infty} \xi(c) = \eta.
\end{equation}

We now focus on the properties of the continuation region $]\bar x_1, \bar x_2[$ and the target states $x^*_i$ with respect to the parameter $c$. All the other parameters are assumed to be fixed. To underline the dependence on this parameter, we will write $\bar x_i = \bar x_i (c)$, $x^*_i = x^*_i (c)$, $A_{ij}=A_{ij}(c)$ and $V_i=V_i^c$, for $i,j \in \{1,2\}$. For the limits, we will write $V^{0^+}_i=\lim_{c \to 0^+} V_i^c$, $x^*_i(+\infty)=\lim_{c \to +\infty} x^*_i(c)$ and so on.

\paragraph{Limits as $c \to 0^+$.} Since we are going to consider the limit $c \to 0^+$ and since by assumption we need $c > \tilde c$, with $\tilde c$ fixed, we assume $\tilde c=0$. If the fixed intervention cost vanishes, that is $c \to 0^+$, we expect that the players continuously intervene to keep the process in a state which satisfies both of them (namely $\tilde s$, for symmetry reasons): in other words, we expect the continuation region $]\bar x_1(c), \bar x_2(c)[$ to collapse to the singleton $\{\tilde s\}$, as $c\to 0^+$. Practically, if the initially state is $x$, either player 1 (if $x < \tilde s$) or player 2 (if $x > \tilde s$) shifts the process to $\tilde s$; from then on, we constantly have $X_s \equiv \tilde s$. As a consequence, we guess that the equilibrium payoff function for player 2 is 
\begin{multline*}
V_2^{0^+}(x) = 
\eee \bigg[ \int_0^{\infty} e^{-\rho s} (s_2 - \tilde s) ds - \lambda (x-\tilde s)\mathbbm{1}_{\{x < \tilde s\}} + \tilde \lambda (\tilde s - x) \mathbbm{1}_{\{x > \tilde s\}}
\bigg] 
\\
= \frac{s_2 - \tilde s}{\rho} - (\tilde \lambda \mathbbm{1}_{\{x > \tilde s\}} + \lambda\mathbbm{1}_{\{x < \tilde s\}})(x - \tilde s) .
\end{multline*}
{\color{black} Notice that this limit situation is formally degenerate with respect our framework, in the sense that it requires singular interventions in continuous time.} We now rigorously prove these heuristic arguments by considering the explicit expression for the intervention region provided in \eqref{explicit}. Actually, the limit situation is not as straightforward as it may appear: the parameters $\lambda, \tilde \lambda$ play an important role. 

\begin{prop}
	Assume $\tilde c =0$ and $\lambda = \tilde \lambda$. Then we have, for $i \in \{1,2\}$ and $x \in \rr$
	\begin{equation*}
	\bar x_i(0^+) = x^*_i(0^+) = \tilde s,
	\qquad\quad
	V^{0^+}_1 (x) = \frac{\tilde s - s_1}{\rho} + \lambda (x - \tilde s),
	\qquad\quad
	V^{0^+}_2 (x) = \frac{s_2 - \tilde s}{\rho} - \lambda (x-\tilde s).
	\end{equation*}
\end{prop}

\begin{proof}
	By \eqref{explicit} and \eqref{limxi} it follows that $\bar x_2(c) \to \tilde s$ as $c \to 0^+$. The same result holds for $\bar x_1$ by symmetry and hence also for $x^*_i$, since $x^*_i \in ]\bar x_1, \bar x_2[$. Moreover,  again by \eqref{explicit} and \eqref{limxi}, we get
	\begin{equation*}
	A_{21}(0^+) = e^{-\theta \tilde s} \frac{\eta}{2\theta}, 
	\qquad\qquad 
	A_{22}(0^+) = - e^{\theta \tilde s} \frac{\eta}{2\theta};
	\end{equation*}
	hence, by the first part of the proof, for each $x \in \rr$ we have (recall that $\lambda = \tilde \lambda$)
	\begin{equation*}
	V^{0^+}_2(x) = 
	\vphi_2^{A_{21}(0^+), A_{22}(0^+)} (\tilde s) - \lambda (x - \tilde s) = \frac{s_2- \tilde s}{\rho} - \lambda (x - \tilde s).
	\end{equation*}
	The corresponding result for $V_1$ follows by symmetry.
\end{proof}

{\color{black} 

\begin{prop}
	Assume $\tilde c =0$ and $\lambda > \tilde \lambda$. Then we have
	\begin{equation*}
	\bar x_1(0^+) = x^*_1(0^+) = \tilde s - \zeta < \tilde s + \zeta = x^*_2(0^+) = \bar x_2(0^+), \qquad \zeta =\frac{1}{\theta} \log\left(\sqrt{\frac{\lambda - \tilde \lambda}{2 \eta}+1} + \sqrt{\frac{\lambda - \tilde \lambda}{2 \eta}}\right) >0.
	\end{equation*}
\end{prop}

\begin{proof}
	The result immediately follows by \eqref{explicit} and \eqref{limxi}.
\end{proof}

The case $\lambda = \tilde \lambda$ corresponds to the intuition above, with $\bar x_i(0^+) = x^*_i(0^+) = \tilde s$, for $i \in \{1,2\}$. Conversely, in the case $\lambda > \tilde \lambda$ we get $\bar x_1(0^+) = x^*_1(0^+)$  and $\bar x_2(0^+) = x^*_2(0^+)$, but the two values, quite surprisingly, do not coincide. We remark that we here have a non-trivial mutual continuation region $]\bar x_1,\bar x_2[ \,=\, ]\tilde s - \zeta, \tilde s + \zeta[$. Practically, when the process exits from such a region, one of the players moves the process to one of the two boundaries, then she continuously intervenes to keep the process in that state, until the Brownian motion points inwards the continuation region. Then, the game goes on. 


}

\paragraph{Limits as $c \to +\infty$.} If the intervention cost increases, the players rarely intervene. In the limit case $c \to +\infty$, they never intervene and we expect $]\bar x_1(c), \bar x_2(c)[$ to coincide with $\rr$. Correspondingly, the state variable diffuses without being affected by the players, that is $X_s = x + \sigma W_s$, for each $s\geq 0$. As a consequence, we guess that the equilibrium payoff function for player 2 is 
\begin{equation*}
V_2^{+\infty}(x) = 
\eee \bigg[ \int_0^{\infty} e^{-\rho s} (s_2 - x - \sigma W_s) ds \bigg] = \frac{s_2 - x}{\rho}.
\end{equation*}
Moreover, the intervening player clearly compensates the cost $c + \lambda|x^*_i -x|$ by moving the process to a state where her payoff is bigger than the opponent's one. In the case $c \to + \infty$, the intervening player has to compensate diverging costs, so that we guess that $x^*_i(c)$ diverges too. We now rigorously prove our guesses. 

\begin{prop}
	The following limits hold: 
	\begin{gather*}
	\bar x_2(+\infty) = x^*_1(+\infty) = + \infty,
	\qquad\qquad
	\bar x_1(+\infty) = x^*_2(+\infty) = - \infty,
	\\
	V^{+\infty}_1 (x) = \frac{x - s_1}{\rho},
	\qquad\qquad
	V^{+\infty}_2 (x) = \frac{s_2 - x}{\rho}.
	\end{gather*}
\end{prop}

\begin{proof}
	By \eqref{explicit} and \eqref{limxi} it easily follows that $\bar x_2(+\infty) = +\infty$ and $x^*_2(+\infty) = -\infty$. By symmetry, corresponding results hold for $\bar x_1, x^*_1$. Moreover, by \eqref{explicit} and \eqref{limxi} we get 
	\begin{equation*}
	A_{21}(+\infty) = A_{21}(+\infty) = 0;
	\end{equation*}
	hence, by the first part of this proof, for each $x \in \rr$ we have
	\begin{equation*}
	V^{+\infty}_2(x) = \vphi_2^{A_{21}(+\infty), A_{22}(+\infty)} (x)= \frac{s_2 - x}{\rho}. 
	\end{equation*}
	The corresponding result for $V_1$ follows by symmetry.
\end{proof}

\paragraph{Monotonicity of $\bar x_i, x^*_i$.} If the intervention cost $c$ increases, we expect the common continuation region $]\bar x_1(c), \bar x_2(c)[$ to enlarge, since the players are less willing to intervene. Proposition \ref{PropMon} makes this guess rigorous.

\begin{prop}
	\label{PropMon}
	The function $c \mapsto \bar x_2(c)$, with $c \in ]\tilde c, +\infty[$, is increasing and the function $c \mapsto \bar x_1(c)$ is decreasing.
\end{prop}

\begin{proof}
	Let us prove that $c \mapsto \bar x_2 (c)$, with $c \in ]\tilde c, +\infty[$, is increasing. By \eqref{explicit} it suffices to check that
	\begin{equation*}
	c \mapsto \frac{\eta + \xi(c)}{\eta - \xi(c)} \left(\frac{\theta(\lambda - \tilde \lambda + \eta)}{4\eta}\frac{c}{\xi(c)}- \frac{\theta\tilde c}{4}\frac{1}{\xi(c)} + \frac{\lambda- \tilde \lambda}{2 \eta}\right)
	\end{equation*}
	is an increasing function. Since $\xi'>0$, a sufficient condition is that
	\begin{equation*}
	c \mapsto \frac{c}{\xi(c)}, \qquad c>\tilde c,
	\end{equation*}
	is increasing, that is 
	\begin{equation*}
	H(c)= \xi(c) - c \xi'(c) \geq 0, \qquad \forall c>\tilde c,
	\end{equation*}
	which is true since by \eqref{limxi} we have $H(\tilde c^+)=0$ (consider separately the cases $\tilde c =0$ and $\tilde c >0$) and $H'(c)>0$. The result for $\bar x_1(c)$ follows by symmetry.
\end{proof}

As for the monotonicity of $x^*_i$, it is not easy to make a guess. The formulas in \eqref{explicit} do not allow easy estimates; however, a monotonicity result can be proved in the case $\tilde c=0$. We will see later by some numerical simulations that in the general case the function $x^*_i$ is not monotone.

\begin{prop}
	Assume $\tilde c = 0$. Then, the function $c \mapsto x^*_2(c)$, with $c \in ]0, +\infty[$, is decreasing and the function $c \mapsto x^*_1(c)$ is increasing. Moreover, we have $x^*_2 < \tilde s < x^*_1$ for each $c>0$.
\end{prop}

\begin{proof}
	Let us prove that $c \mapsto x^*_2 (c)$, with $c \in ]0, +\infty[$, is decreasing. By \eqref{explicit} it suffices to prove that
	\begin{equation*}
	c \mapsto \frac{\theta(\lambda - \tilde \lambda + \eta)}{4\eta}\frac{c(\eta - \xi(c))}{\xi(c) (\eta + \xi(c))} + \frac{\lambda- \tilde \lambda}{2 \eta}\frac{\eta - \xi(c)}{\eta + \xi(c)}
	\end{equation*}
	is a decreasing function. Since $\xi'>0$, a sufficient condition is that
	\begin{equation*}
	c \mapsto \frac{c(\eta - \xi(c))}{\xi(c) (\eta + \xi(c))}
	\end{equation*}
	is decreasing, that is
	\begin{equation*}
	K(c)= \xi(c) - \frac{\theta}{2} \frac{c(\eta^2 -\xi^2(c))}{\xi^2(c)} - \theta \eta \frac{c}{\xi(c)}\leq 0, \qquad \forall c>0,
	\end{equation*}
	which is true since by \eqref{limxi} we have $K(0^+)=0$ and $K'(c)<0$. The results for $x^*_1(c)$ follows by symmetry. Finally, we get the inequalities by $(x^*_2)'<0 < (x^*_1)'$ and $x^*_1(0^+) = x^*_2(0^+) = \tilde s$.
\end{proof}

\paragraph{Limits as $c \to \tilde c^+$.} We conclude this section with the behaviour as $c \to \tilde c^+$ in the case $\lambda = \tilde \lambda$, i.e.~when each intervention practically becomes a transfer of money from the intervening player to the opponent. It is not easy to guess what happens in this case and the result is quite surprising: the limiting strategies are not admissible.

\begin{prop}
	Assume $\lambda = \tilde \lambda$. For $i,j \in \{1,2\}$ with $i \neq j$, the following limits hold:
	\begin{equation*}
	\bar x_i(\tilde c^+) = x^*_j(\tilde c^+) = \tilde s + \frac{(-1)^i}{2\theta}\log \left(\frac{\eta + \xi(\tilde c)}{\eta - \xi(\tilde c)} \right).
	\end{equation*}
\end{prop}

\begin{proof}
	The result immediately follows by \eqref{explicit} and \eqref{limxi}.
\end{proof}

Essentially, the limit situation is as follows. Let $i,j \in \{1,2\}$ with $i \neq j$; as soon as the process reaches $\bar x_i$, player $i$ moves the process to $x^*_i = \bar x_j$, which is the boundary of the intervention region of player $j$, who moves the process back to $\bar x_i$, thus causing another intervention by player $i$ and so on. We get a infinite sequence of simultaneous interventions, meaning that these strategies are not admissible.

\paragraph{Numerical simulations.} Here, we present the results (obtained with Wolfram Mathematica) of some numerical simulations on the game we have described. We focus on player 2 and consider the following two sets of parameters:
\begin{align*}
&\text{Problem 1: $\rho = 0.02$,  $\sigma = 0.15$, $s_1 = -3$, $s_2 = 3$, $\tilde c = 0$, $\lambda = \tilde \lambda=15$}.
\\
&\text{Problem 2: $\rho = 0.02$, $\sigma = 0.15$, $s_1 = -3$, $s_2 = 3$, $\tilde c = 50$, $\lambda = \tilde \lambda=0$}.
\end{align*}
We here consider the Nash equilibrium corresponding to $\tilde s = (s_1 + s_2)/2 = 0$. 

Figure \ref{PICex1} represents the equilibrium payoff function $x \mapsto V^c_2(x)$ for Problem 1 and $c=100$ (the dashed lines correspond to the three components of the function). Similarly, in Figure \ref{PICex2} we plot the function $x \mapsto V^c_2(x)$ for Problem 2 and $c=100$. In both cases, we notice the $C^1$-pasting in $\bar x_2$, whereas, as noticed in Section \ref{sec:stochImpGame}, the functions are not differentiable in $\bar x_1$. Also, when $\lambda$ is non-zero, the function is unbounded. 

Figure \ref{PICint1} (for Problem 1, with $c \in \,\, ]\tilde c, \infty[ \, = \, ]0, \infty[$) and Figure \ref{PICint2} (for Problem 2, with $c \in \,\, ]\tilde c, \infty[ \, = \, ]50, \infty[$) show the continuation region and the target states: namely, we plot $c \mapsto \bar x_1(c)$ (solid blue line), $\bar x_2(c)$ (solid green line), $x^*_1(c)$ (dashed blue line) and $x^*_2(c)$ (dashed green line). As proved above, the continuation region enlarges as $c$ grows and diverges as $c \to \infty$. Consider the limit case $c \to \tilde c^+$: if $\tilde c=0$, the four parameters converge to the same state $\tilde s$; conversely, in the case $\tilde c >0$, we see that $x^*_1$ (resp.~$x^*_2$) converges to $\bar x_2$ (resp.~$\bar x_1$), which corresponds to an inadmissible game. Also, we notice that $x^*_1$ (resp.~$x^*_2$) is decreasing (resp.~increasing) when $\tilde c =0$, whereas such functions are not monotone in the $\tilde c >0$ case.

We finally consider Problem 1 and the evolution of $x \mapsto V_2^c(x)$ as $c$ grows: Figure \ref{PICevo1} corresponds to $c=0$, Figure \ref{PICevo2} to $c=250$, Figure \ref{PICevo3} to $c=500$. The equilibrium payoff function is a straight line in the limit case $c \to 0^+$, then a bell-shaped curve appears; as $c$ grows, the local maximum moves to the left and the right side of the bell resembles more and more a straight line with slope $1/\rho$, which is actually the limit as $c \to + \infty$.

\begin{figure}[htb]
	\centering
	\begin{minipage}{0.45\textwidth}
		\centering
		\includegraphics[width=\textwidth]{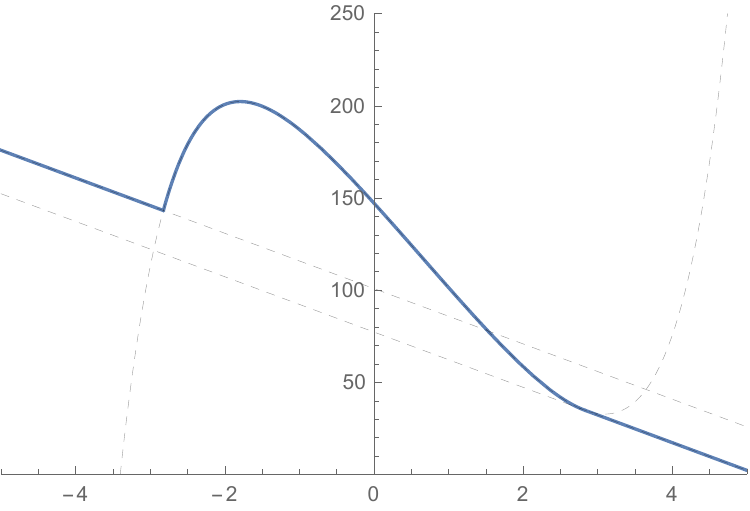}
		\caption{$x \mapsto V_2^c(x)$ for Pr.~1 and $c=100$}
		\label{PICex1}
	\end{minipage}
	\begin{minipage}{0.05\textwidth}
		~
	\end{minipage}
	\begin{minipage}{0.45\textwidth}
		\centering
		\includegraphics[width=\textwidth]{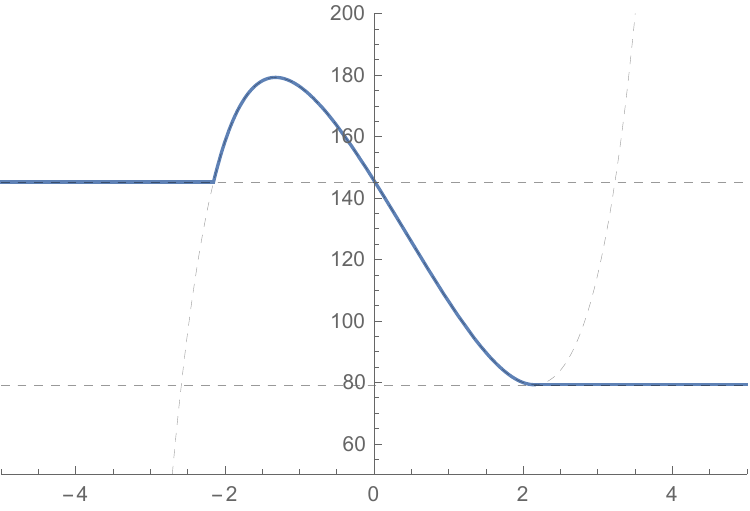}
		\caption{$x \mapsto V_2^c(x)$ for Pr.~2 and $c=100$}
		\label{PICex2}
	\end{minipage}
\end{figure}

\begin{figure}[htb]
	\centering
	\begin{minipage}{0.45\textwidth}
		\centering
		\includegraphics[width=\textwidth]{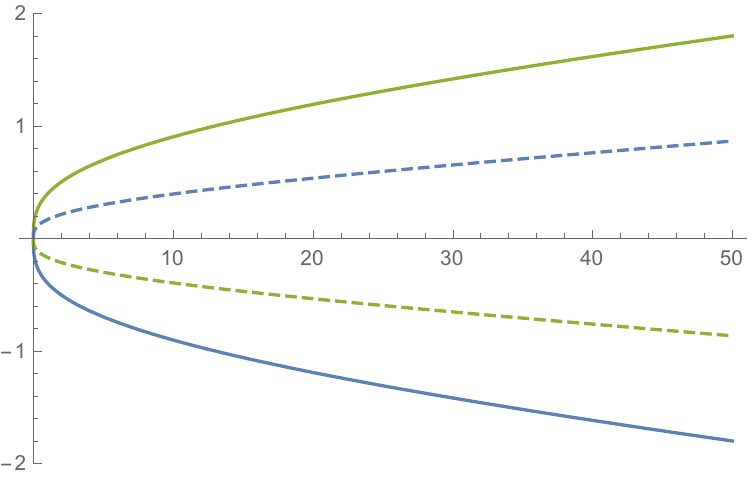}
		\caption{$c \mapsto \bar x_i(c), x^*_i(c)$ for Problem 1}
		\label{PICint1}
	\end{minipage}
	\begin{minipage}{0.05\textwidth}
		~
	\end{minipage}
	\begin{minipage}{0.45\textwidth}
		\centering
		\includegraphics[width=\textwidth]{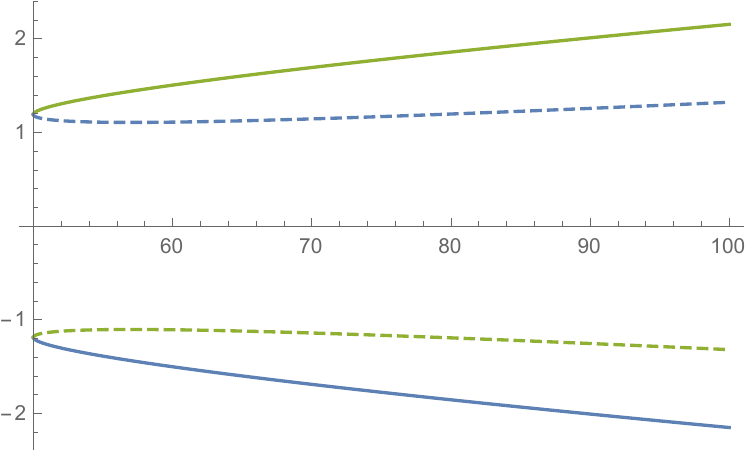}
		\caption{$c \mapsto \bar x_i(c), x^*_i(c)$ for Problem 2}
		\label{PICint2}
	\end{minipage}
	\vspace{0.5cm} 
\end{figure}

\begin{figure}[!htb]
	\centering
	\begin{minipage}{0.32\textwidth}
		\centering
		\includegraphics[width=\textwidth]{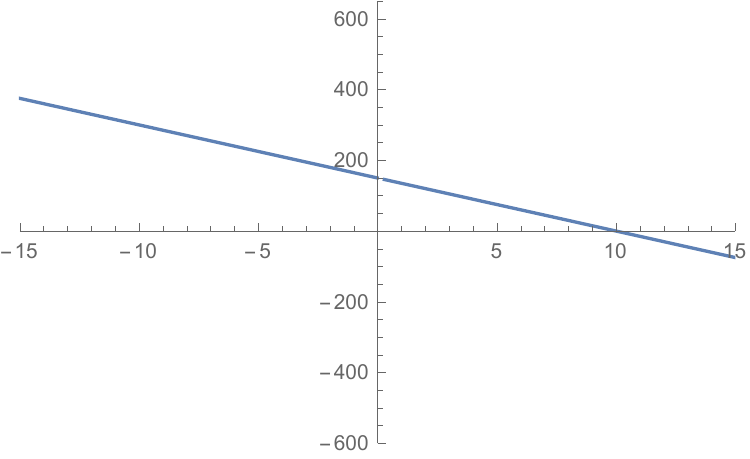}
		\caption{Prob.~1, $c=0$}
		\label{PICevo1}
	\end{minipage}
	\begin{minipage}{0.32\textwidth}
		\centering
		\includegraphics[width=\textwidth]{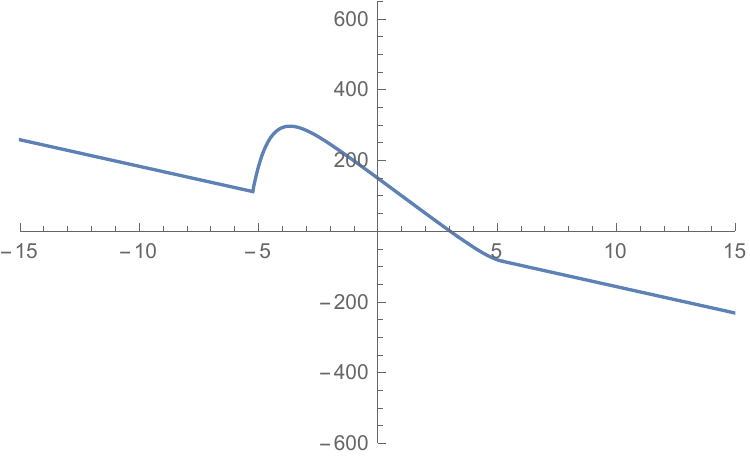}
		\caption{Prob.~1, $c=250$}
		\label{PICevo2}
	\end{minipage}
	\begin{minipage}{0.32\textwidth}
		\centering
		\includegraphics[width=\textwidth]{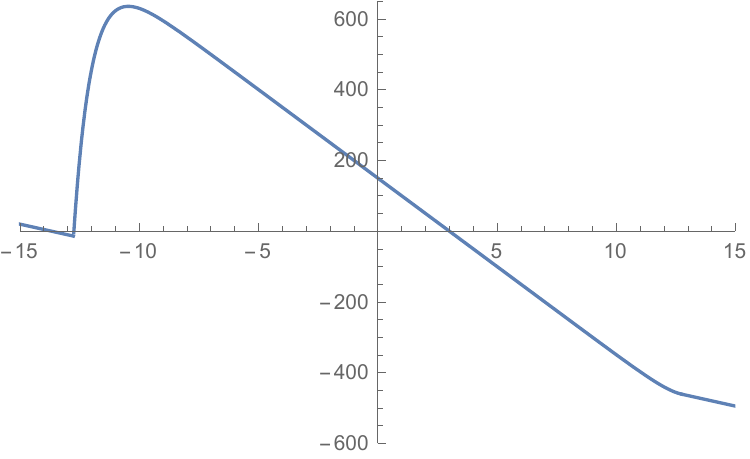}
		\caption{Prob.~1, $c=500$}
		\label{PICevo3}
	\end{minipage}
\end{figure}

\begin{remark}
{\em
	Recall that the family of Nash equilibria in Proposition 	\ref{prop:checkverification} is parametrized by $\tilde s \in \rr$. In this section, we have focused on the properties of $V_1,V_2$ with respect to the cost parameter $c>0$, assuming a fixed value for $\tilde s$. By the formulas in Remark \ref{RemForm}, we also get some properties for the equilibrium payoff functions in the case where $\tilde s$ varies and the other parameters are fixed. To underline the dependence on $\tilde s$, we now write $\bar x_i= \bar x_i(\tilde s)$ and $V_i=V_i^{\tilde s}$. As $\tilde s \in \rr$ increases, by \eqref{explicit} the common continuation region $]\bar x_1(\tilde s), \bar x_2 (\tilde s)[$ moves to the right, which corresponds to smaller values for Player 2, since $f_2$ is decreasing. Indeed, it is easy to see from Definition \ref{def:candidate} and \eqref{explicit} that for any $x \in \rr$ the function $\tilde s \mapsto V_2^{\tilde s}(x)$ is decreasing, with limits $V_2^{+\infty}(x)=-\infty$ and $V_2^{-\infty}(x)=+\infty$. Similar results hold for $V_1^{\tilde s}$.
}
\end{remark}

\subsection{Further examples}
\label{ssec:moreexamples}

The arguments in Sections \ref{sec:candidate} and \ref{sec:applying} can be easily adapted to other problems: we here provide some further examples. Clearly, in most cases one has to deal with the full 8-equation system \eqref{OrdCondTWOPL}-\eqref{syst1}-\eqref{syst2} (the decoupling technique in Proposition \ref{PropEx} is only possible with symmetric payoffs) and the solution has to be found numerically.

\paragraph{Cubic payoffs.} Let us consider the same setting as in Section \ref{sec:thepb}, now with cubic payoffs: namely, we substitute \eqref{ExPayoff} with 
\begin{equation*}
\tilde f_1(x) = c_1(x - s_1)^3, \qquad \tilde f_2(x)= (s_2-x)^3, \qquad s_1<s_2, 
\end{equation*}
for $x \in \rr$, where $c_1$ is a strictly positive constant. 

To find an expression for the Nash equilibrium, we follow the procedure introduced in the previous sections, as follows. First, we solve \eqref{EqEx}, with $f_i$ substituted by $\tilde f_i$: for $x \in \rr$, the solutions are given by
\begin{gather*}
\tilde \vphi_1(x) = A_{11} e^{\q x} +  A_{12} e^{- \q x} + \frac{c_1}{\rho}(x-s_1)^3 + \frac{3c_1\sigma^2}{\rho^2} (x-s_1), 
\\
\tilde \vphi_2 (x) = A_{21} e^{\q x} + A_{22} e^{-\q x} + \frac{1}{\rho}(s_2-x)^3 + \frac{3\sigma^2}{\rho^2} (s_2-x).
\end{gather*}
Then, by the same arguments as in Section \ref{sec:candidate}, a pair of (candidate) equilibrium payoff functions is given by \eqref{Vpratic}, with $\vphi_1,\vphi_2$ substituted by $\tilde \vphi_1,\tilde \vphi_2$. In order to have a well-posed definition, we have to find a solution $(A_{ij}, \bar x_i, x^*_i)_{i,j \in \{1,2\}}$ to the 8-equation system \eqref{OrdCondTWOPL}-\eqref{syst1}-\eqref{syst2}. If $c_1 = 1$, we can apply the symmetry argument in the proof of Proposition \ref{PropEx} and consider a reduced system with four equations. In general, however, we need to deal with the full 8-equation system. In both cases, the solution has to be found numerically. Finally, given a solution to \eqref{OrdCondTWOPL}-\eqref{syst1}-\eqref{syst2}, we have to verify that the candidates actually satisfy all the assumptions of the Verification Theorem \ref{thm:verification}. We proceed as in Section \ref{sec:applying}: the only difference occurs when verifying that 
\begin{equation}
\label{temp}
\begin{gathered}
-\rho \big( \tilde \vphi_1(\bar x_1) - \lambda (\bar x_1-x) \big) + \tilde f_1(x) \leq 0, \qquad \text{$\forall x \in [-\infty, \bar x_1[$,}
\\
-\rho \big( \tilde \vphi_2(\bar x_2) - \lambda (x-\bar x_2) \big) + \tilde f_2(x) \leq 0, \qquad \text{$\forall x \in [\bar x_2, \infty[$.}
\end{gathered}
\end{equation}
Indeed, here we cannot use the monotonicity argument in the proof of Proposition \ref{prop:checkverification}, so that \eqref{temp} has to be checked numerically. Provided that \eqref{temp} holds, we can then conclude that a Nash equilibrium exists: player 1 (resp.~player 2) intervenes when the state variable exits from $]\bar x_1,+\infty[$ (resp.~$]-\infty,\bar x_2[$) and shifts the process to $x^*_1$ (resp.~$x^*_2$). 

As an example, we consider the following values:
\begin{equation*}
\rho = 0.1, \quad
\sigma = 0.2, \quad
c = 60, \quad
\tilde c = 20, \quad
\lambda = \tilde \lambda = 5, \quad
s_1 = -3, \quad
s_2 = 3, \quad
c_1 = 1.2.
\end{equation*}
A solution to \eqref{OrdCondTWOPL}-\eqref{syst1}-\eqref{syst2}, which also satisfies \eqref{temp}, is numerically given by 
\begin{gather*}
A_{11} = -104.943, \qquad
A_{12} = 12.965, \qquad
A_{21}= 24.669, \qquad
A_{22}= -56,001, 
\\ 
x^*_1 = 0.186, \qquad
x^*_2 = -0.453, \qquad
\bar x_1= -0.732, \qquad
\bar x_2 =0.464.
\end{gather*}
Notice that the continuation region $]\bar x_1, \bar x_2[$ is closer to $s_1$ than to $s_2$, which is reasonable: since $c_1>1$, player 1 experiences higher gains and losses if compared to player 2, so that she is more willing to intervene than her opponent, which practically translates into $|\bar x_1 - s_1| < |\bar x_2 - s_2|$.

\paragraph{Linear and cubic payoffs.} Let us consider the same setting as in Section \ref{sec:thepb}, but now player 1 has a cubic payoff and player 2 has a linear payoff: namely, we substitute \eqref{ExPayoff} with 
\begin{equation*}
\hat f_1(x) = c_1(x - s_1)^3, \qquad \hat f_2(x)= s_2-x, \qquad s_1<s_2, 
\end{equation*}
for $x \in \rr$, where $c_1$ is a strictly positive constant. 

As above, a pair of candidate equilibrium payoff functions is given by \eqref{Vpratic}, with $\vphi_1,\vphi_2$ substituted by
\begin{gather*}
\hat\vphi_1(x) = A_{11} e^{\q x} +  A_{12} e^{- \q x} + \frac{c_1}{\rho}(x-s_1)^3 + \frac{3c_1\sigma^2}{\rho^2} (x-s_1), 
\\
\hat\vphi_2 (x) = A_{21} e^{\q x} + A_{22} e^{-\q x} + \frac{1}{\rho}(s_2-x).
\end{gather*}
Provided that a solution $(A_{ij}, \bar x_i, x^*_i)_{i,j \in \{1,2\}}$ to the systems \eqref{OrdCondTWOPL}-\eqref{syst1}-\eqref{syst2} exists (with $\vphi_1,\vphi_2$ substituted by $\hat \vphi_1,\hat \vphi_2$), and that 
\begin{gather}
\label{temp2}
-\rho \big( \hat \vphi_1(\bar x_1) - \lambda (\bar x_1-x) \big) + \hat f_1(x) \leq 0, \qquad \text{$\forall x \in [-\infty, \bar x_1[$,}
\end{gather}
then a Nash equilibrium exists and is described as in the previous example. Notice that in \eqref{temp2} we do not need a sign condition for player 2, since this is implied by \eqref{ExCond}, as in the proof of Proposition \ref{prop:checkverification}. 

As an example, we consider the following values:
\begin{equation*}
\rho = 0.1, \quad
\sigma = 0.2, \quad
c = 10, \quad
\tilde c = 50, \quad
\lambda = \tilde \lambda = 0, \quad
s_1 = -0.5, \quad
s_2 = 0.5, \quad
c_1 =1.
\end{equation*}
A solution to \eqref{OrdCondTWOPL}-\eqref{syst1}-\eqref{syst2}, which also satisfies \eqref{temp2}, is numerically given by 
\begin{gather*}
A_{11} = -4.886, \qquad
A_{12} = 0.739, \qquad
A_{21}= 0.418, \qquad
A_{22}= -0.713, 
\\ 
x^*_1 = 0.752, \qquad
x^*_2 = -0.814, \qquad
\bar x_1= -1.319, \qquad
\bar x_2 = 1.053.
\end{gather*}
For the same reasons as above, we remark that $]\bar x_1, \bar x_2[$ is closer to $s_1$ than to $s_2$.

\begin{remark}
\emph{For the sake of simplicity, in this section we have kept the same dynamics as in the previous section, when none of the players intervenes, that is $dX_s = \sigma dW_s$. However, different equations could be considered, as this would only affect the definition of the coefficient $\theta$ in $\vphi_1,\vphi_2$.}
\end{remark}

\section{Conclusions}
\label{sec:conclusion}

In this paper, we have considered a general two-player nonzero-sum impulse game, whose state variable follows a diffusive dynamics driven by a multi-dimensional Brownian motion. After setting the problem, we have provided a verification theorem giving sufficient conditions in order for the solutions of a suitable system of quasi-variational inequalities to coincide with the payoff functions of the two players at some Nash equilibrium. To the best of our knowledge this result is new to the literature on impulse games and it constitutes the major mathematical contribution of the present paper. As an application, we have provided a solvable one-dimensional impulse game where two players with linear running payoffs can shift a real-valued Brownian motion in order to maximize their objective functions. We have found a family of Nash equilibria and explicitly characterized the corresponding equilibrium strategies. We have also studied some asymptotic properties of the Nash equilibria. As a final contribution, we have considered two further families of examples, with cubic payoffs and with linear and cubic payoffs, where a solution is found numerically.


\appendix

\section{Appendix}
\label{SecApp}

\begin{proof}[Proof of Lemma \ref{lemmaprocess}]
	\textit{Step 1.} We prove that \eqref{pbTEMP} implies \eqref{pbCOSSO}. The only property to be proved is (\ref{pbCOSSO-qvi}). We consider three cases. First, assume $V = \widetilde \mm V$. Since $\aaa V + f \leq 0$ and $\widetilde \mm V - V = 0$, we have $\max \{ \aaa V + f, \widetilde \mm V - V \} =0$, which implies (\ref{pbCOSSO-qvi}) since $\widetilde \hh V - V \geq 0$. Then, assume $\widetilde \mm V < V < \widetilde \hh V$. Since $\aaa V + f = 0$ and $\widetilde \mm V - V < 0$, we have $\max \{ \aaa V + f, \widetilde \mm V - V \} =0$, which implies (\ref{pbCOSSO-qvi}) since $\widetilde \hh V - V > 0$. Finally, assume $V = \widetilde \hh V$. Since $\aaa V + f \geq 0$ and $\widetilde \mm V - V \leq 0$, we have $\max \{ \aaa V + f, \widetilde \mm V - V \} \geq 0$, which implies (\ref{pbCOSSO-qvi}) since $\widetilde \hh V - V = 0$.
	
	\textit{Step 2.} We show that \eqref{pbCOSSO} implies \eqref{pbTEMP}. The only properties to be proved are (\ref{pbTEMP-M}), (\ref{pbTEMP-MH}) and (\ref{pbTEMP-H}). We assume  $\widetilde \mm V < \widetilde \hh V$ (the case $\widetilde \mm V = \widetilde \hh V$ being immediate) and consider three cases. First, assume $V = \widetilde \mm V$. Since $\widetilde \hh V - V > 0$, from (\ref{pbCOSSO-qvi}) it follows that $\max \{ \aaa V + f, 0 \} = 0$, which implies $\aaa V + f \leq 0$. Then, assume $\widetilde \mm V < V < \widetilde \hh V$. Since $\min \{ \max \{ \alpha, \beta \}, \gamma \} \in \{\alpha, \beta, \gamma\}$ for every $\alpha, \beta, \gamma \in \rr$, and since $\widetilde \mm V - V < 0 < \widetilde \hh V - V$, from (\ref{pbCOSSO-qvi}) it follows that $ \aaa V + f = 0$. Finally, assume $V = \widetilde \hh V$. From (\ref{pbCOSSO-qvi}) it follows that $\max \{ \aaa V + f, \widetilde \mm V - V \} \geq 0$, which implies $\aaa V + f \geq 0$ since $\widetilde \mm V - V < 0$.
\end{proof}

\vspace{0.2cm}

{\color{black}
\begin{proof}[Complements to the proof of Proposition \ref{PropEx}.] \textit{(a)} Let us prove that the conditions in \eqref{OrdCondTWOPL} and \eqref{syst2} are equivalent to the system in \eqref{news}. Let $\eta = (1-\lambda\rho)/\rho$ and $\tilde \eta = (1-\tilde \lambda\rho)/\rho$. By the definition of $\vphi_2$ in \eqref{defxi} and the change of variable in \eqref{Eq0}, the order condition \eqref{OrdCondTWOPL} and the system \eqref{syst2} write
\begin{subnumcases}{\label{App1}}
A_1 (y^*)^2 - 2 \eta y^* - A_2 = 0, \label{App1-a} \\ 
A_1 \bar y^2 - 2 \eta \bar y - A_2 = 0, \label{App1-b} \\ 
A_1 \Big(\frac{1}{\bar y} - \frac{1}{y^*}\Big) + A_2 (\bar y - y^*) - 2\theta \tilde c + 2 \tilde \eta \log(\bar y / y^*)=0, \label{App1-c} \\
A_1 (\bar y - y^*) + A_2  \Big(\frac{1}{\bar y} - \frac{1}{y^*}\Big) + 2\theta c - 2 \eta  \log(\bar y / y^*)=0, \label{App1-d}\\
y^*>0, \quad \bar y>0, \quad y^*<\bar y, \quad 1<\bar y y^*, \quad A_1(y^*)^2 + A_2 \leq 0. \label{App1-e}
\end{subnumcases}
In particular, as for the conditions in \eqref{App1-e}: we need $y^*,\bar y>0$ by \eqref{Eq0}, the inequalities $y^*<\bar y$ and $1<\bar y y^*$ correspond to \eqref{OrdCondTWOPL}, the condition $A_1(y^*)^2 + A_2 \leq 0$ corresponds to $\vphi_2''(x^*_2) \leq0$. Now, notice that $A_1(y^*)^2 + A_2 = 2y^*(A_1y^* + \eta)$ by the equation in \eqref{App1-a}. Moreover, by \eqref{App1-a} and \eqref{App1-b} we have 
\begin{equation*}
\frac{1}{y^*} = \frac{A_1}{A_2} y^* - \frac{2\eta}{A_2}, \qquad\qquad
\frac{1}{\bar y} = \frac{A_1}{A_2} \bar y - \frac{2\eta}{A_2},
\end{equation*}
which are well-defined since $A_2 \neq 0$ (indeed, $A_2=0$ and \eqref{App1-a}-\eqref{App1-b} would imply either $y^*=0$ or $\bar y=0$, in contradiction with \eqref{App1-e}). Hence, the system in \eqref{App1} can be rewritten as
\begin{subnumcases}{\label{App2}}
A_1 (y^*)^2 - 2 \eta y^* - A_2 = 0, \label{App2-a} \\ 
A_1 \bar y^2 - 2 \eta \bar y - A_2 = 0, \label{App2-b} \\ 
\frac{A_1^2+A_2^2}{A_2}(\bar y - y^*) - 2\theta \tilde c + 2 \tilde \eta \log(\bar y / y^*) =0, \label{App2-c} \\
A_1 (\bar y - y^*) + \theta c - \eta  \log(\bar y / y^*) =0, \label{App2-d}\\
y^*>0, \quad \bar y>0, \quad y^*<\bar y, \quad 1<\bar y y^*, \quad A_1y^* - \eta \leq 0. \label{App2-e}
\end{subnumcases}
We finally get \eqref{news} by substituting \eqref{App2-c} with the sum of \eqref{App2-c} and \eqref{App2-d}.

\textit{(b)} Let us prove that \eqref{proof5} is equivalent to \eqref{proof6}. By the definition of $\xi$ and \eqref{proof5-b}, we have
\begin{equation*}
\sqrt{\eta^2 + A_1A_2}=\xi,
\end{equation*}
which immediately leads to
\begin{equation*}
A_1A_2=\xi^2-\eta^2 = -M,
\end{equation*}
with $M>0$ as in \eqref{proof6}. Notice that $\eta^2+A_1A_2 >0$ is then always verified. Moreover, by using \eqref{proof5-b} to rewrite the logarithm and by the relations $\sqrt{\eta^2 + A_1A_2}=\xi$ and $A_1A_2=\xi^2-\eta^2$, the equation in \eqref{proof5-a} can be rewritten as
\begin{equation*}
(A_1+A_2)^2 \xi + A_1 A_2 \left[\theta(c-\tilde c) + (\lambda - \tilde \lambda) \frac{2 \xi + \theta c}{\eta}\right] =0.
\end{equation*}
By \eqref{proof5-c} we then have
\begin{equation*}
A_1+A_2=-\sqrt{\frac{(\eta^2 - \xi^2)\big[\theta\eta(c-\tilde c) + (\lambda -  \tilde\lambda) (2\xi + \theta c)\big]}{\eta\xi}} = -2N,
\end{equation*}
with $N>0$ as in \eqref{proof6}. Notice that $A_1+A_2 <0$ is then always satisfied.
\end{proof}
}

\end{document}